\documentclass[11pt]{amsart}
\usepackage[utf8]{inputenc}

\usepackage{amsmath}
\usepackage{amssymb}
\usepackage{amstext}
\usepackage{amsthm}
\usepackage{tikz}
\usetikzlibrary{calc}
\usepackage{scalefnt}
\usepackage{algorithmic}
\usepackage{algorithm}
\usepackage{comment}

\makeatletter

\newcommand{\N}{\ensuremath{\mathbb{N}}}
\newcommand{\R}{\ensuremath{\mathbb{R}}}
\newcommand{\Prob}{\ensuremath{\mathbb{P}}}

\newcommand{\calC}{\mathcal{C}}
\newcommand{\calD}{\mathcal{D}}
\newcommand{\calF}{\mathcal{F}}
\newcommand{\calH}{\mathcal{H}}
\newcommand{\calI}{\mathcal{I}}
\newcommand{\calJ}{\mathcal{J}}
\newcommand{\calK}{\mathcal{K}}
\newcommand{\calM}{\mathcal{M}}

\newcommand{\calO}{\mathcal{O}}
\newcommand{\calP}{\mathcal{P}}

\newcommand{\calU}{\mathcal{U}}
\newcommand{\calR}{\mathcal{R}}
\newcommand{\calS}{\mathcal{S}}
\newcommand{\calX}{\mathcal{X}}

\newcommand{\E}{\mathbb{E}}

\newcommand{\op}[1]{\operatorname{#1}}
\newcommand{\norm}[2][\relax]{\ifx#1\relax \ensuremath{\left\Vert#2\right\Vert}\else \ensuremath{\left\Vert#2\right\Vert_{#1}}\fi}
\newcommand{\BIGOP}[1]{\mathop{\mathchoice
{\raise-0.22em\hbox{\huge $#1$}}
{\raise-0.05em\hbox{\Large $#1$}}{\hbox{\large $#1$}}{#1}}}
\newcommand{\bigtimes}{\BIGOP{\times}}

\newcommand{\dd}{\mathrm{d}}
\newcommand{\bs}[1]{\boldsymbol{#1}}
\newcommand{\isdef}{\mathrel{\mathrel{\mathop:}=}}

\newcommand{\de}{\dd}
\newcommand{\balpha}{{\boldsymbol\alpha}}
\newcommand{\bgamma}{{\boldsymbol\gamma}}
\newcounter{thm}

\theoremstyle{plain}
\newtheorem{theorem}[thm]{Theorem}
\newtheorem{lemma}[thm]{Lemma}

\theoremstyle{definition}
\newtheorem{remark}[thm]{Remark}

\newtheorem{definition}[thm]{Definition}

\begin{document}

\title{Multilevel tensor approximation of PDEs with random data}

\author{Jonas Ballani}
\author{Daniel Kressner}
\address{Jonas Ballani, Daniel Kressner,
MATHICSE-ANCHP, \'Ecole Polytechnique F\'ed\'erale de Lausanne, Station 8, CH-1015 Lausanne.\\
}
\email{\{jonas.ballani,daniel.kressner\}@epfl.ch}
\author{Michael Peters}
\address{
Michael Peters,
Department of Mathematics and Computer Science, University of Basel, Spiegelgasse 1, CH-4051 Basel\\
}
\email{michael.peters@unibas.ch}
\thanks{The first author has been supported by an EPFL fellowship through the European Union's Seventh Framework Programme under grant agreement no.~291771.}

\maketitle

\begin{abstract}
In this paper, we introduce and analyze a new low-rank multilevel strategy for the solution of random diffusion problems.
Using a standard stochastic collocation scheme, we first approximate the infinite dimensional random problem
by a deterministic parameter-dependent problem on a high-dimensional parameter domain. Given a hierarchy of finite element discretizations
for the spatial approximation, we make use of a multilevel framework in which we consider the differences of the solution
on two consecutive finite element levels in the collocation points. We then address the approximation of these 
high-dimensional differences by adaptive low-rank tensor techniques. This allows to equilibrate the error on all levels
by exploiting analytic and algebraic properties of the solution at the same time. We arrive at 
an explicit representation in a low-rank tensor format of the approximate solution on the entire parameter domain, which can be used for, e.g., the direct
and cheap computation of statistics. Numerical results are provided in order to illustrate the approach.
\end{abstract}

\section{Introduction}
In this article, we consider the random boundary value problem
\begin{equation}\label{eq:theproblem}
    -\nabla\cdot\big(a(\omega)\nabla u(\omega)\big) = f\text{ in } D,\quad u(\omega)=0\text{ on } \partial D,
\end{equation}
where \(D\subset\mathbb{R}^d\) denotes a domain and \(\omega\in\Omega\) is a random parameter, with \(\Omega\) denoting the set of possible outcomes. As the solution \(u\) 
depends on the parameter $\omega$, we aim at an
efficient approximation of the solution map $\omega \mapsto u(\omega)$. The numerical solution of~\eqref{eq:theproblem} has attracted quite some attention during the last decade, motivated by the need for quantifying the impact of uncertainties in PDE-based models.

The key idea of our novel approach is to combine a multilevel stochastic collocation framework with adaptive low-rank tensor techniques. This involves the following steps:
\begin{enumerate}
\item A standard technique for random diffusion problems, the Karhunen-Lo\`eve expansion 
of the diffusion coefficient $a$ is truncated after $N\in\N$ terms to turn~\eqref{eq:theproblem} into a parametric PDE depending on $N$ random parameters. This truncated problem is approximated by a stochastic collocation scheme.
\item We use a hierarchy of finite element discretizations for discretizing the physical domain $D$ and represent the solution $u$ as a telescoping sum.
The smoothness properties of the solution $u$ are exploited to adapt the polynomial degrees in the stochastic collocation 
of the differences of $u$ between two consecutive finite element levels.
This allows us to choose a low polynomial degree for the fine spatial discretization while using higher polynomial
degrees only on coarser finite element levels. 
\item Because of the high dimensionality of the parameter domain, each difference in the multilevel sum needs to be evaluated
in a large number of collocation points. We use adaptive low-rank tensor techniques to obtain good approximations from a relatively small number of samples. This allows us to exploit the algebraic structure of the solution with respect to the random parameters automatically
while maintaining the accuracy of the multilevel scheme.
\end{enumerate}

Both, multilevel and low-rank tensor approximation techniques, have been extensively studied for the solution of~\eqref{eq:theproblem}.
In the following, we briefly describe some of the existing approaches.

A number of different multilevel techniques have been proposed that aim at equilibrating the errors of the spatial approximation and the approximation in the random parameter. 
If a statistics of the solution or a quantity of interest needs to be computed, multilevel quadrature methods,
like the multilevel (quasi-)Monte Carlo method or even more general quadrature approaches, are 
feasible; we refer to~\cite{BSZ,G,GW,GriHarPet2015,HPS1,H,H2} for instances of this approach.
Closer to the setting considered in this paper, the work by Teckentrup et al.~\cite{TJWG} proposes to directly interpolate the solution \(u\) itself
in suitable collocation points in the parameter domain from a sparse index set. Given additional smoothness in the spatial variable, a spatial sparse-grid approximation can be incorporated, which leads to the multiindex stochastic collocation proposed in~\cite{ANTT}.

Low-rank tensor approximation techniques have turned out be a versatile tool for solving PDEs with random data; see~\cite{Nouy2015,Nouy2015b} and the references therein. In particular, a variety of low-rank approaches have been proposed to address the linear systems arising from a Galerkin discretization of~\eqref{eq:theproblem}; see, e.g.,~\cite{DolKLM2015,DooI2009,EigPS2015,EspHackLitMatWae2012,EspHackLitMatZan2013,KhoOsel2010,KhoSchwa2011,KresTob2011,MatZan2012,Tob2012}. Non-intrusive tensor-based approaches for uncertainty quantification can be built upon black box approximation techniques~\cite{BaGr2014,BaKr2016,Oseledets2010,Savostyanov2011}.

To the best of our knowledge, there is little work on merging multilevel and tensor approximation techniques in uncertainty quantification. Recently, Lee and Elman~\cite{LeeE2016} proposed a two-level scheme in the context of the Galerkin method for PDEs with random data. This scheme uses the solution from the coarse level to identify a dominant subspace in the domain of the random parameter, which in turn is used to speed up the solution on the fine level by avoiding costly low-rank truncations. The combination of multilevel and tensor approximation techniques proposed in this paper is conceptually different and is not 
restricted to this two level approach but allows for multiple levels.

The rest of this paper is organized as follows. In Section~\ref{sec:problem}, we formulate the mathematical setting and recall the Karhunen-Lo\`eve expansion. 
Section~\ref{sec:discretization} is concerned with the discretization of~\eqref{eq:theproblem} in the physical and the stochastic domain. In Section~\ref{sec:multilevel}, we describe an existing multilevel scheme and analyze the impact of perturbations on this scheme. Section~\ref{sec:lowrank} contains the main contribution of this paper, a novel combination of the multilevel scheme with a low-rank tensor approximation. Finally, Section~\ref{sec:experiments} reports numerical results for PDEs with a random diffusion coefficient on the unit square featuring a variety of different stochastic diffusion coefficients.

Throughout this article, in order to avoid the repeated use of generic but unspecified constants, we indicate by \(C\lesssim D\) that \(C\) can be bounded by a multiple of \(D\), independently of parameters which \(C\) and \(D\) may depend on. Obviously, \(C\gtrsim D\) is defined as \(D\lesssim C\), and we write \(C\eqsim D\) if there holds 
\(C\lesssim D\) and \(C\gtrsim D\).

\section{Problem Setting} \label{sec:problem}
Let $D\subset \R^d$ denote a bounded Lipschitz domain. Typically, we have \(d=2,3\).
Moreover, let $(\Omega,\calF,\Prob)$ be a complete and separable probability space, where $\Omega$ is the set of 
outcomes, $\calF\subset 2^\Omega$ is the $\sigma$-algebra of events, and $\Prob\colon\calF\to[0,1]$ is a probability 
measure on \(\calF\).
We are interested in solving the following stochastic diffusion problem: Find $u\in L^2\big(\Omega;H^1_0(D)\big)$ such that
\begin{equation*}
\begin{aligned}
    -\nabla\cdot\big(a(\omega)\nabla u(\omega)\big) &= f,&&\text{in } D,\\
    u(\omega) &= 0,&&\text{on }\partial D,
\end{aligned}
\end{equation*}
holds \(\Prob\)-almost surely. 
Here and in the sequel, for a
Banach space  \(\mathcal{X}\), we define the \emph{Lebesgue-Bochner-space} \(L^p(\Omega;\mathcal{X})\), $1\le p\le\infty$
as the space of all equivalence classes of
strongly measurable functions $v\colon\Omega\to \mathcal{X}$ whose norm
\[
  \|v\|_{L^p(\Omega;\mathcal{X})} \isdef 
        \begin{cases} \displaystyle{\left(\int_\Omega \|v(\omega)\|_{\mathcal{X}}^p 
                \de\mathbb{P}(\omega)\right)^{1/p}},& p<\infty\\[2ex]
        \displaystyle{\operatorname*{ess\,sup}_{\omega\in\Omega}\|v(\omega)\|_{\mathcal{X}}},& p=\infty\end{cases}
\]
is finite. If $p=2$ and $\mathcal{X}$ is a separable Hilbert 
space, then the Bochner space is isomorphic to the tensor 
product space 
\[
L^2(\Omega;\calX)\cong L^2(\Omega)\otimes\calX.
\]

Throughout this article, we shall assume that the load \(f\in L^2(D)\) is purely deterministic.
Still, by straightforward modifications it is also possible to deal with random load vectors, see, e.g.,~\cite{BabNobTem2010}. 
Additionally, for the sake of simplicity, we restrict ourselves here to the case of uniformly elliptic diffusion problems. This means
that we assume the existence of constants $a_\mathrm{min} > 0$ and $a_\mathrm{max} < \infty$ that are independent
of the parameter \(\omega\in\Omega\) such that for almost every \(x\in D\) there holds
\begin{equation}\label{eq:unifcond}
a_\mathrm{min} \leq a(\omega,x) \leq a_\mathrm{max}\quad\Prob\text{-almost surely}.
\end{equation}
Nevertheless, we emphasize that the presented approach is directly transferable to diffusion problems, where the 
constants $a_\mathrm{min}$ and $a_\mathrm{max}$ might depend on \(\omega\in\Omega\) and are only \(\Prob\)-integrable,
as it is the case for log-normally distributed diffusion coefficients, cf.~\cite{HS11,SG}.
Therefore, all results presented here remain valid in this case.

Typically, the diffusion coefficient is not directly feasible for numerical computations and has thus to be represented in
a suitable way.
To that end, one decomposes the diffusion coefficient
with the aid of the \emph{Karhunen-Lo\`eve expansion}. 

Let
the covariance kernel of \(a(\omega,x)\) be defined by the positive semi-definite function
\[
\mathcal{C}(x,x')\isdef\int_\Omega \big(a(\omega,x)-\E[a]({x})\big)\big(a(\omega,{x}')-\E[a]({x}')\big)
\de\mathbb{P}.
\]
Herein, the integral with respect to \(\Omega\) has to be understood in terms of a Bochner integral, cf.~\cite{HP57}.
One can show that \(\mathcal{C}(x,x')\) is well defined if there holds \(a\in L^2(\Omega;\mathcal{X})\).
Now, let \(\{(\lambda_n,\varphi_n)\}_n\) denote the eigenpairs obtained by solving the
eigenproblem for the diffusion coefficient's covariance, i.e.~
\[
\int_D\mathcal{C}({x},{x}')\varphi_n({x}')\de{x}'=\lambda_n\varphi_n({x}).
\]
Then, the Karhunen-Lo\`eve expansion of \(a(\omega,x)\) is given by
\begin{equation}\label{eq:KLE}
a(\omega,x)
     =\E[a]({x})+\sum_{n=1}^\infty\sqrt{\lambda_n}X_{n}(\omega)\varphi_n({x}),
\end{equation}
where \(X_n\colon\Omega\to\Gamma_n\subset\mathbb{R}\) for \(n=1,2,\ldots\) are centered, pairwise uncorrelated and
\(L^2\)-normalized
random variables given by
\[
X_n\isdef\frac{1}{\sqrt{\lambda_n}}\int_D \big(a(\omega,x)-\E[a](x)\big)\varphi_n(x)\de x.
\] 
From condition \eqref{eq:unifcond}, we directly infer, that the image of the random variables is a bounded set and
that \(\E[a]({\bf x})>0\). Thus, without loss of generality, we assume that \(\Gamma_n=[-1,1]\).
The important cases, which we wish to study here, are the uniformly distributed case, i.e.~\(X_n\sim\mathcal{U}([-1,1])\)
and the log-uniformly distributed case which means that we have diffusion coefficient of the form
\(\exp\big(a(\omega,x)\big)\), where \(a(\omega,x)\) is given as in the uniformly distributed case and
satisfies \eqref{eq:unifcond}.

Although, we have separated by now the spatial and the stochastic influences in the diffusion coefficient, 
we are still facing an infinite sum. Nevertheless, for numerical issues, this sum may be truncated appropriately.
The impact of truncating the Karhunen-Lo\`eve expansion on the solution
is bounded by
\[
\|u-u_{N}\|_{L^2(\Omega;H^1_0(D))}\lesssim \|a-a_{N}\|_{L^2(\Omega;L^\infty(D))}=\varepsilon(N),
\]
where \(\varepsilon(N)\to 0\) montonically as \(N\to\infty\), see e.g.~\cite{Cha12,ST06}. 
Herein,
we set
\[
a_N(\omega,x)
     \isdef\E[a]({x})+\sum_{n=1}^N\sqrt{\lambda_n}X_{n}(\omega)\varphi_n({x}),
\]
and \(u_N\) is the solution to
\[
-\nabla\cdot\big(a_N(\omega)\nabla u_N(\omega)\big) = f\text{ in } D,\quad u(\omega) = 0\text{ on }\partial D.
\]
Note that these estimates relate to the log-normal and the uniformly distributed cases. But they also directly transfer
to the log-uniform case.

Assuming additionally, that the \(\{X_n\}_n\) are independent and exhibit densities 
\(\rho_n\colon\Gamma_N\to\mathbb{R}_{+}\) with respect to
the Lebesgue measure,
we end up, with the parametric diffusion problem: Find \(u_N\in L^2_\rho\big(\Gamma;H^1_0(D)\big)\)
\begin{equation}\label{eq:paramEQ}
-\nabla\cdot\big(a_N(\bs{y})\nabla u_N(\bs{y})\big) = f\text{ in } D,
\end{equation}
where \(\rho\isdef\rho_1(y_1)\cdots\rho_N(y_N)\), \(\Gamma\isdef\bigtimes_{n=1}^N \Gamma_n\)
and \(\bs{y}=\bs{y}(\omega)\isdef [y_1(\omega),\ldots,y_N(\omega)] \in \Gamma\).
Herein, the space \(L^2_\rho\big(\Gamma;H^1_0(D)\big)\) is endowed with the norm
\[
\norm{v}_{L^2_\rho(\Gamma;H^1_0(D))}\isdef \left( \int_{\Gamma} \norm{v(\bs{y})}_{H^1_0(D)}^2 \rho(\bs{y}) \dd \bs{y} \right)^{1/2}.
\]
Note that we have \(\rho_n=1/2\) for the case of \(X_n\sim\mathcal{U}([-1,1])\).
In view of the polynomial interpolation with respect to the parameter \(\bs{y}\in\Gamma\), we shall finally introduce for a
Banach space \(\mathcal{X}\) the
space
\[
C^0(\Gamma;\mathcal{X})\isdef\Big\{v\colon\Gamma\to\mathcal{X}: v\text{ is continuous and }\sup_{\bs{y}\in\Gamma}\|v(\bs{y})\|_{\mathcal{X}}<\infty\Big\}.
\]

\section{Discretization} \label{sec:discretization}

Later on, a standard \emph{stochastic collocation} scheme, cf.~\cite{BabNobTem2010}, is used for the stochastic discretization of
the differences of the solutions to the parametric diffusion problem~\eqref{eq:paramEQ} on consecutive grids. 
To that end, we use tensor product polynomial interpolation
in the parameter space $\Gamma$ and a finite element approximation in the physical domain $D$.

\subsection{Polynomial Interpolation}
Let $\calP_{\bs{p}}(\Gamma) \subset L_\rho^2(\Gamma)$ denote the span of tensor product polynomials with degree at most $\bs{p}=(p_1,\ldots,p_N)$, i.e.,
\[
    \calP_{\bs{p}}(\Gamma) = \bigotimes_{n=1}^N \calP_{p_n}(\Gamma_n)
\]
with
\[
    \calP_{p_n}(\Gamma_n) = \operatorname{span}\{ y_n^m : m=0,\ldots,p_n\},\quad n=1,\ldots,N.  
\]
Given interpolation points $y_{n,k_n}\in \Gamma_n$, $k_n=0,\ldots,p_n$, the Lagrange basis for $\calP_{p_n}(\Gamma_n)$ is defined by $\{l_{n,k_n}\in\calP_{p_n}(\Gamma_n):l_{n,k_n}(y_{n,j_n}) = \delta_{k_n,j_n}, k_n =0,\ldots,p_n \}$.
By a tensor product construction, we obtain the Lagrange basis $\{l_{\bs{k}}\}$ for \(\calP_{\bs{p}}(\Gamma)\) where
\[
    l_{\bs{k}}(\bs{y}) \isdef \prod_{n=1}^N l_{n,k_n}(y_n)
\]
for a multiindex $\bs{k}=(k_1,\ldots,k_N)\in\calK_{\bs{p}}$
with 
\[
    \calK_{\bs{p}}\isdef\{(k_1,\ldots,k_N)\in \N_0^N : k_n = 0,\ldots,p_n,\, n=1,\ldots,N\}. 
\]

For all functions \(v\in C^0\big(\Gamma;H^1_0(D)\big)\), the tensor product interpolation points\\ $\bs{y}_{\bs{k}} \isdef (y_{1,k_1},\ldots,y_{N,k_N}) \in \Gamma$ give rise to the interpolation operator
\[
    \calI_{\bs{p}}\colon C^0\big(\Gamma; H_0^1(D)\big) \to \calP_{\bs{p}}(\Gamma) \otimes H_0^1(D)
\]
defined by
\begin{equation}\label{eq:ipol_def}
    \calI_{\bs{p}}[v](\bs{y}) = \sum_{\bs{k}\in\calK_{\bs{p}}} v(\bs{y}_{\bs{k}}) l_{\bs{k}}(\bs{y}).
\end{equation}

With regard to \eqref{eq:paramEQ}, our goal is to approximate the solution $u_N$ by
\[
u_N(\bs{y}) \approx \calI_{\bs{p}}[u_N](\bs{y}) = \sum_{\bs{k}\in\calK_{\bs{p}}} u_N(\bs{y}_{\bs{k}}) l_{\bs{k}}(\bs{y}).
\]
In order to obtain the coefficients $u_N(\bs{y}_{\bs{k}})$, we have to solve 
\begin{equation}\label{eq:CollcoeffEQ}
-\nabla\cdot\big(a_N(\bs{y}_{\bs{k}})\nabla u_N(\bs{y}_{\bs{k}})\big) = f\text{ in } D,\quad u_N(\bs{y}_{\bs{k}})=0\text{ on }\partial D,
\end{equation}
for all $\bs{y}_{\bs{k}}$ with $\bs{k}\in\calK_{\bs{p}}$. For each $\bs{k}\in\calK_{\bs{p}}$, \eqref{eq:CollcoeffEQ} is a deterministic diffusion problem on $D$ which can be approximated by the finite element method.

\subsection{Interpolation Error}


To study the impact of the interpolation error, we have to take the smoothness of \(u_N\) with respect to the parameter \(\bs{y}\in\Gamma\) into account.
It is well known, see, e.g.,~\cite{CDS,HS11}, that \(u_N\) satisfies the decay estimate
\begin{equation}\label{eq:decEst1}
\big\|\partial^{\balpha}_{\bs{y}}u_N(\bs{y})\big\|_{H^1_0(D)}\leq C|\balpha|!c^{|\balpha|}\bgamma^\balpha\|f\|_{L^2(D)},\text{ where }\gamma_n\isdef\sqrt{\lambda_n}\|\varphi\|_{L^\infty(D)},
\end{equation}
cf.~\eqref{eq:KLE},
for some constants \(C,c>0\). In the sequel, we consider the interpolation based on the Chebyshev nodes
\[
\eta_k\isdef \cos\bigg(\frac{2k+1}{2(p+1)}\pi\bigg)\in[-1,1],\quad k=0,\ldots,p.
\]
The related uni-directional interpolation operator shall be denoted by 
\[\mathcal{I}_p\colon C([-1,1])\to\mathcal{P}_p,\quad v(x)\mapsto \sum_{k=0}^pv(\eta_k) l_k(x).\]
It satisfies for a function \(v\in C^{p+1}([-1,1])\) the well known interpolation error estimate
\[
\bigg|v(x)-\sum_{k=0}^p v(\eta_k)l_k(x)\bigg|\leq\frac{1}{2^{p}(p+1)!}\max_{\xi\in[-1,1]}\big|v^{(p+1)}(\xi)\big|
\]
and the stability estimate
\[
\bigg\|\sum_{k=0}^p v(\eta_k)l_k(x)\bigg\|_{C^0([-1,1])}\leq\bigg(\frac{2}{\pi}\log(p+1)+1\bigg)\|v\|_{C^0([-1,1])}, 
\]
see, e.g.,~\cite{Riv74}. Therefore, we obtain by tensor product construction the stability estimate for \(\calI_{\bs{p}}\)
according to
\[
\big\|\calI_{\bs{p}}[v]\big\|_{C^0(\Gamma; H_0^1(D))}\leq C_s(\bs{p})\|v\|_{C^0(\Gamma; H_0^1(D))}
\]
with
\[
C_s(\bs{p})\isdef\prod_{i=1}^N\bigg(\frac{2}{\pi}\log(p_i+1)+1\bigg).
\]
Obviously, the stability constant will grow exponentially as \(N\to\infty\). Nevertheless, this case is not considered here.
Moreover, we emphasize that there exist regimes, where the stability constant is bounded. If the error is, for example,
measured in \(L^2_\rho(\Gamma; H_0^1(D))\) and the interpolation points are chosen as the roots of the orthogonal polynomials
with respect to the densities \(\rho_n\), then the corresponding stability estimate holds with \(C_s(\bs{p})=1\), cf.~\cite{BabNobTem2010}.
Still, without the orthogonality property, there exist also bounds of the stability constant for Chebyshev points, 
if the error is measured in \(L^1(\Gamma; H_0^1(D))\), see~\cite{Fej33}.
Nevertheless, in order to obtain a black box interpolation, which is independent of the particular density function, 
we will employ here the Chebyshev points and measure the error with respect to \(C^0\big(\Gamma; H_0^1(D)\big)\) at the cost
of a stability constant that is not robust with respect to the polynomial degree.

Thus, we obtain the following interpolation result for the solution \(u_N\) to \eqref{eq:paramEQ}, which is a straightforward modification
of the related result in \cite{HB02}.

\begin{theorem}\label{thm:interpErr}
Let \(c\gamma_k<2\). Then, given that
\[
p_k=\bigg\lceil\frac{\log(\varepsilon)}{\log(c\gamma_k/2)}\bigg\rceil -1,
\]
the polynomial interpolation satisfies the error estimate
\[
\big\|u_N({\bs{y}})-\calI_{\bs{p}}[u_N](\bs{y})\big\|_{H^1_0(D)}\lesssim \varepsilon C(\bs{p})\|f\|_{L^2(D)}
\]
for some constant $C(\bs{p})$.
\end{theorem}
\begin{proof}
There holds by \eqref{eq:decEst1} and the repeated application of the triangle inequality  that
\begin{align*}
&\big\|u_N({\bs{y}})-\calI_{\bs{p}}[u_N](\bs{y})\big\|_{H^1_0(D)}\\
&\quad\leq\sum_{k=1}^N
\big\|\big(\calI_{p_1}\otimes\ldots\otimes\calI_{p_{k-1}}\otimes(\op{Id}-{\calI_{p_k}})
\otimes \op{Id}\otimes\ldots\otimes\op{Id}\big)u_N({\bs{y}}\big)\big\|_{H^1_0(D)}\\
&\quad\leq\sum_{k=1}^N\bigg[\prod_{m=1}^{k-1}\bigg(\frac{2}{\pi}\log (p_m+1)+1\bigg)\bigg]
\bigg[\frac{1}{2^{p_k}(p_k+1)!}C(p_k+1)!c^{p_k+1}\gamma_k^{p_k+1}\bigg]\|f\|_{L^2(D)}\\
&\quad=\sum_{k=1}^N\bigg[\prod_{m=1}^{k-1}\bigg(\frac{2}{\pi}\log(p_m+1)+1\bigg)\bigg]
\bigg[2\bigg(\frac{c\gamma_k}{2}\bigg)^{p_k+1}C c\bigg]\|f\|_{L^2(D)}.
\end{align*}
Thus, with 
\[p_k=\bigg\lceil\frac{\log(\varepsilon)}{\log(c\gamma_k/2)}\bigg\rceil-1,
\]
we obtain
\[
\big\|u_N({\bs{y}})-\calI_{\bs{p}}[u_N](\bs{y})\big\|_{H^1_0(D)}
\leq Cc\varepsilon\Bigg(\sum_{k=1}^N\bigg[\prod_{m=1}^{k-1}\bigg(\frac{2}{\pi}\log(p_m+1)+1\bigg)\bigg]\Bigg)\|f\|_{L^2(D)}.
\]
\end{proof}

\begin{remark}
The constant \(C(\bs{p})\) from the previous theorem can be bounded according to
\[
C_s(\bs{p})\leq C(\bs{p})\leq(N+1)C_s(\bs{p}).
\]
where we recall that \(C_s(\bs{p})\) denotes the stability constant of \(\calI_{\bs{p}}\).
Thus, \(C(\bs{p})\) also potentially grows exponentially as \(N\to\infty\).

\end{remark}
\subsection{Finite Element Approximation}
In order to compute the coefficients \(u_N(\bs{y}_{\bs{k}})\) in $\eqref{eq:CollcoeffEQ}$, we consider an approximation by the finite element method.
To this end, let $\mathcal{T}_0=\{\tau_{0,k}\}$ be a coarse grid triangulation
of the domain $D$. Then, for $\ell
\ge 1$, a uniform and shape regular triangulation
$\mathcal{T}_\ell=\{\tau_{\ell,k}\}$ is recursively obtained by uniformly 
refining each element $\tau_{\ell-1,k}$ into $2^d$ 
elements with diameter $h_\ell\eqsim 2^{-\ell}$.
We define the space of piecewise linear finite elements according to
\begin{equation}\label{eq:def_FE_space}
\mathcal{S}_\ell^1(D)\isdef\{v\in C(D):v|_{\partial D}=0\ \text{and}\ 
	v|_\tau\in\Pi_1\ \text{for all}\ \tau\in\mathcal{T}_\ell\}\subset H_0^1(D),
\end{equation}
where $\Pi_1$ denotes the space of all polynomials of total 
degree $1$. Then, the finite element approximations 
$u_{N,\ell}({\bs{y}_{\bs{k}}})\in\mathcal{S}_\ell^1(D)$ to the coefficients 
\(u_N({\bs{y}_{\bs{k}}})\) satisfy the following well known error estimate.
\begin{lemma}
Let the domain $D$ be convex and $f\in L^2(D)$. Then, for \(\bs{y}\in\Gamma\),
the finite element solution $u_{N,\ell}(\bs{y})\in\mathcal{S}_\ell^1(D)$
of the diffusion problem \eqref{eq:paramEQ} satisfies the error 
estimate
\begin{equation}	\label{eq:err-est-1}
  \|u_N(\bs{y})-u_{N,\ell}(\bs{y})\|_{H^1(D)}\lesssim 2^{-\ell}\|u_N(\bs{y})\|_{H^2(D)}\lesssim 2^{-\ell}\|f\|_{L^2(D)}.
\end{equation}
\end{lemma}
Note that we restrict ourselves here to the situation of piecewise linear
finite elements. Nevertheless, by applying obvious modifications, the
presented results remain valid also for higher order finite elements.
Moreover, for the sake of simplicity, we consider here nested sequences
of finite element spaces, i.e.,
\begin{equation}\label{eq:nested_FEM}
\mathcal{S}_0^1(D)\subset \mathcal{S}_1^1(D)\subset\ldots.
\end{equation}
This is not a requirement, as has been discussed in \cite{GriHarPet2015}.

\subsection{Stochastic Collocation Error}
By a tensor product argument, the combination of the finite element approximation 
in the spatial variable
and the interpolation in the parameter yields the following approximation
result.
\begin{theorem} Let the polynomial degree \(\bs{p}\)
be chosen such that there holds 
\[
\big\|u_N({\bs{y}})-\calI_{\bs{p}}[u_N](\bs{y})\big\|_{H^1_0(D)}\lesssim 2^{-\ell}C(\bs{p})\|f\|_{L^2(D)},
\]
where \(u_{N,\ell}(\bs{y})\) is the finite element approximation to \(u_N(\bs{y})\)
on level \(\ell\) that fulfills \eqref{eq:err-est-1}.
Then, the fully discrete approximation \(\calI_{\bs{p}}[u_{N,\ell}]\in\mathcal{P}_{\bs{p}}(\Gamma)\otimes\mathcal{S}^1_j(D)\)
satisfies the error estimate
\[
\big\|u_N({\bs{y}})-\calI_{\bs{p}}[u_{N,\ell}](\bs{y})\big\|_{H^1_0(D)}\lesssim 2^{-\ell}\big(C(\bs{p})+C_s(\bs{p})\big)\|f\|_{L^2(D)},
\]
where \(C_s(\bs{p})\) denotes the stability constant of \(\calI_{\bs{p}}\).
\end{theorem}

%

\section{Multilevel Approximation} \label{sec:multilevel}
In the previous section, we have introduced the classical stochastic collocation as it has been proposed in, e.g., 
\cite{BabNobTem2010}. The related error estimate is in this case based on a tensor product argument between
the spatial approximation and the discretization of the parameter. Now, the idea of the related multilevel 
approximation is to perform an error equilibration as it is known from \emph{sparse} tensor product approximations,
cf.~\cite{BG}.

\subsection{Multilevel Scheme}
We start by representing the finite element approximation \(u_{N,L}(\bs{y})\) for a maximal level $L\in\N$ by the telescoping sum
\[
u_{N,L}(\bs{y})=\sum_{\ell=0}^L \big(u_{N,\ell}(\bs{y})-u_{N,\ell-1}(\bs{y})\big)\quad\text{with }u_{N,-1}\isdef 0.
\]
Instead of applying the interpolation operator in the parameter \(\bs{y}\in\Gamma\) with a fixed degree \({\bs{p}}\),
we adapt the degree to the finite element approximation level and obtain the multilevel approximation
\begin{equation}\label{eq:MLrep}
u_N(\bs{y})\approx u_{N,L}^{\operatorname{M\!L}}(\bs{y})\isdef\sum_{\ell=0}^L \calI_{\bs{p}^{(\ell)}}\big[u_{N,\ell}-u_{N,\ell-1}\big](\bs{y}).
\end{equation}
The goal is now to choose the polynomial degrees \(\{\bs{p}^{(\ell)}\}\) antipodal to the refinement level $\ell$ of the
finite element approximation and to equilibrate a high spatial accuracy with a relatively low polynomial degree.
In order to facilitate this, it is crucial to have the following \emph{mixed regularity} estimate for \(u_N\).
There holds
\[
\big\|\partial^{\balpha}_{\bs{y}}u(\bs{y})\big\|_{H^2(D)}\leq C|\balpha|!c^{|\balpha|}\tilde{\bgamma}^\balpha\|f\|_{L^2(D)},\text{ where }\tilde{\gamma}_k\isdef\sqrt{\lambda_k}\|\varphi_k\|_{W^{1,\infty}(D)},
\]
cp.~\eqref{eq:KLE}, for some constants \(C,c>0\).
See \cite{CDS} for a proof of this statement in the affine case and \cite{HS11} for the log-normal case.
The estimate for the log-uniform case can be derived with the same techniques that are applied in these works.
From this estimate, one can derive the parametric smoothness of the Galerkin error. This is stated by the following lemma
which is, e.g.,~proven in \cite{HPS2,Kuo2015}.
\begin{lemma}\label{lem:decay}
For the error of the Galerkin projection, 
there holds the estimate
\begin{equation*}\label{eq:mixest}
\big\|\partial^{\balpha}_{\bs{y}}(u_N-u_{N,\ell})(\bs{y})\big\|_{H^1(D)}\lesssim
2^{-\ell}|\balpha|!
c^{|\balpha|}{\bgamma}^{\balpha}\|f\|_{L^2(D)}
\quad\text{for all $|\boldsymbol\alpha|\ge 0$}
\end{equation*}
with a constant \(c>0\) depending on \(a_{\min}\) and \(a_{\max}\),
where \(\gamma_k\isdef\|\sqrt{\lambda_k}\varphi_k\|_{W^{1,\infty}(D)}\) from \eqref{eq:KLE} and \(\bgamma\isdef(\gamma_1,\ldots,\gamma_m)\).
\end{lemma}
With this lemma at hand, it is easy to derive the following error estimate in complete analogy to the proof of Theorem~\ref{thm:interpErr}.
\begin{theorem}
Let the degree \(\bs{p}^{(\ell')}\in\mathbb{N}^N\) be such that the interpolation error is \(C\big(\bs{p}^{(\ell')}\big)\varepsilon\eqsim2^{-\ell'}\).
Then, there holds the error estimate
\begin{equation}\label{eq:genest}
\big\|(\operatorname{Id}-\calI_{\bs p^{(\ell')}})[u_N-u_{N,\ell}](\bs{y})\big\|_{H^1(D)}\lesssim 2^{-(\ell+\ell')}\|f\|_{L^2(D)}.
\end{equation}
\end{theorem}
\begin{theorem}\label{theo:MLest}
Let $\big\{\bs{p}^{(\ell')}\big\}\in\mathbb{N}^N$ be a sequence of 
polynomial degrees, that give rise to an error estimate of
the form \eqref{eq:genest} with \(\ell'=L-\ell\), where
\(u_N\in L^2_\rho\big(\Gamma,H^1_0(D)\big)\) is the solution to 
\eqref{eq:paramEQ} that satisfies \eqref{eq:err-est-1}. 
Then, the error of the multilevel approximation \eqref{eq:MLrep}
is bounded by
\begin{equation}\label{eq:errormlest}
\bigg\|u_N(\bs{y})-\sum_{\ell=0}^L \calI_{\bs{p}^{(\ell)}}\big[u_{N,\ell}-u_{N,\ell-1}\big](\bs{y})\bigg\|_{H^1_0(D)}
   		\lesssim 2^{-L}L\|f\|_{L^{2}(D)}.
\end{equation}
\end{theorem}

\begin{proof}
We shall apply the following multilevel splitting of the error 
\begin{equation}\label{eq:multsplitting}
\begin{aligned}
&\bigg\|u_N(\bs{y})-\sum_{\ell=0}^L\calI_{\bs{p}^{(\ell)}}\big[u_{N,\ell}-u_{N,\ell-1}\big](\bs{y})\bigg\|_{H^1_0(D)}\\
&\quad=\bigg\|u_N(\bs{y})-u_{N,L}(\bs{y})+\sum_{\ell=0}^L(u_{N,\ell}-u_{N,\ell-1})(\bs{y})-\sum_{\ell=0}^L \calI_{\bs{p}^{(\ell)}}\big[u_{N,\ell}-u_{N,\ell-1}\big](\bs{y})\bigg\|_{H^1_0(D)}\\
&\quad\le\big\|u_N(\bs{y})-u_{N,L}(\bs{y})\big\|_{H^1_0(D)}
	+ \sum_{\ell=0}^L\big\|(\operatorname{Id}-\calI_{\bs p^{(\ell')}})[u_{N,\ell}-u_{N,\ell-1}](\bs{y})\big\|_{H^1_0(D)}. 
\end{aligned}
\end{equation}
The first term just reflects the finite element approximation error and is thus bounded
by \eqref{eq:err-est-1}. Thanks to \eqref{eq:genest}, the
term inside the sum satisfies 
\begin{equation*}
\begin{aligned}
&\big\|(\operatorname{Id}-\calI_{\bs p^{(\ell')}})[u_{N,\ell}-u_{N,\ell-1}](\bs{y})\big\|_{H^1_0(D)}\\
&\qquad\leq
\big\|(\operatorname{Id}-\calI_{\bs p^{(\ell')}})[u_{N}-u_{N,\ell}](\bs{y})\big\|_{H^1_0(D)}
+\big\|(\operatorname{Id}-\calI_{\bs p^{(\ell')}})[u_{N}-u_{N,\ell-1}](\bs{y})\big\|_{H^1_0(D)}\\
&\qquad\lesssim
2^{-(\ell+L-\ell)}\|f\|_{L^2(D)}+2^{-(\ell-1+L-\ell)}\|f\|_{L^2(D)}\lesssim 2^{-L}\|f\|_{L^2(D)}.
\end{aligned}
\end{equation*}
Thus, we can estimate \eqref{eq:multsplitting} as
\begin{align*}
\bigg\|u_N(\bs{y})-\sum_{\ell=0}^L \calI_{\bs{p}^{(\ell)}}\big[u_{N,\ell}-u_{N,\ell-1}\big](\bs{y})\bigg\|_{H^1_0(D)}&\lesssim 2^{-L}\|f\|_{L^2(D)}
+\sum_{\ell=0}^L2^{-L}\|f\|_{L^2(D)}\\
&\leq 2^{-L}(L+2)\|f\|_{L^2(D)}.
\end{align*}
This completes the proof.
\end{proof}

\subsection{Perturbed Multilevel Scheme}
The multilevel scheme from above relies on the exact evaluation of the differences \[\delta_{N,\ell} \isdef u_{N,\ell}-u_{N,\ell-1}\] in the interpolation points $\bs{y}_{\bs{k}}^{(\ell)}\in\calK_{\bs{p}^{(\ell)}}\subset\Gamma$ on each level $\ell$. We now slightly relax this assumption and consider perturbations
\begin{equation}\label{eq:perturbed_pts}
  \tilde{\delta}_{N,\ell,\bs{k}} \approx \delta_{N,\ell}(\bs{y}_{\bs{k}}^{(\ell)}),\qquad \bs{k}\in\calK_{\bs{p}^{(\ell)}},
\end{equation}
and the associated perturbed interpolation
\[
    \tilde{\delta}_{N,\ell}(\bs{y}) \isdef \sum_{\bs{k}\in\calK_{\bs{p}^{(\ell)}}} \tilde{\delta}_{N,\ell,\bs{k}} l_{\bs{k}}(\bs{y}).
\]
In view of \eqref{eq:MLrep}, the perturbed multilevel approximation then reads
\begin{equation}\label{eq:perturbed_ML}
\tilde{u}_{N,L}^{\operatorname{M\!L}}(\bs{y})\isdef\sum_{\ell=0}^L  \tilde{\delta}_{N,\ell}(\bs{y}).
\end{equation}
For each level $\ell$, we have the stability estimate
\begin{equation*}
\| \tilde{\delta}_{N,\ell}(\bs{y}) - \calI_{\bs{p}^{(\ell)}}[\delta_{N,\ell}](\bs{y}) \|_{H^1(D)} \leq C_s(\bs{p}^{(\ell)}) \max_{\bs{k}\in\calK_{\bs{p}^{(\ell)}}} \|
\tilde{\delta}_{N,\ell,\bs{k}} -\delta_{N,\ell}(\bs{y}_{\bs{k}}^{(\ell)})
\|_{H^1(D)}.
\end{equation*}
From Theorem \ref{theo:MLest}, we immediately derive the following lemma.
\begin{lemma}\label{lem:perturbed_ML}
Assume that for all $\bs{k}\in\calK_{\bs{p}^{(\ell)}}$ the perturbations from \eqref{eq:perturbed_pts} fulfill
\[
\norm{\tilde{\delta}_{N,\ell,\bs{k}} -\delta_{N,\ell}(\bs{y}_{\bs{k}}^{(\ell)})}_{H^1(D)} \lesssim 2^{-L}\|f\|_{L^2(D)}.
\]
Then
\[
  \norm{u_N(\bs{y}) - \tilde{u}_{N,L}^{\operatorname{M\!L}}(\bs{y})}_{H^1(D)} \lesssim 2^{-L}L\norm{f}_{L^2(D)}.
\]
\end{lemma}
A particular perturbation based on low-rank truncations will be considered in the following section.


\section{Low-Rank Tensor Approximation} \label{sec:lowrank}
The main computational challenge in the multilevel scheme presented above is the evaluation of the differences $\delta_{N,\ell}(\bs{y}_{\bs{k}}^{(\ell)})$ for all $\bs{k}\in\calK_{\bs{p}^{(\ell)}}$. To address high parameter dimensions $N$, we suggest to
approximate these differences in a low-rank tensor format.

Let $n_\ell \isdef \operatorname{dim} \calS_\ell^1(D)$ for the finite element space from \eqref{eq:def_FE_space} and let
$\{\psi_{\ell,i} \in\calS_\ell^1(D): i=1,\ldots,n_\ell\}$ be an orthonormal basis of $\calS_\ell^1(D)$
with respect to the $H_0^1$ inner product, i.e.,
\[
  \langle \psi_{\ell,i}, \psi_{\ell,j} \rangle_{H_0^1(D)} = 0, \qquad i\neq j,
\]
and $\norm{\psi_{\ell,i}}_{H_0^1(D)} = 1$.
Given the nestedness assumption \eqref{eq:nested_FEM}, we have
$\delta_{N,\ell}(\bs{y})\in \calS_\ell^1(D)$. We can hence write 
\begin{equation}\label{eq:delta_psi}
  \delta_{N,\ell}(\bs{y}) = \sum_{i=1}^{n_\ell} \mathbf{u}_i^{(\ell)}(\bs{y}) \psi_{\ell,i}
\end{equation}
with $\mathbf{u}^{(\ell)}(\bs{y}) \in \R^{n_\ell}$. Let now $K_\ell \isdef \#\calK_{\bs{p}^{(\ell)}}$ and define $\mathbf{X}^{(\ell)}\in\R^{K_\ell \cdot n_\ell}$ as
\begin{equation}\label{eq:defX}
    \mathbf{X}_{(\bs{k},i)}^{(\ell)} \isdef \mathbf{u}_i^{(\ell)}(\bs{y}_{\bs{k}}^{(\ell)}),\qquad \bs{k}\in \calK_{\bs{p}^{(\ell)}}.
\end{equation}
In the following, we interpret the vector $\mathbf{X}^{(\ell)}$ as a tensor of order $N+1$ and size 
\[
(p_1^{(\ell)}+1)\times \cdots \times (p_N^{(\ell)}+1) \times n_\ell
\] 
and use low-rank tensor methods to construct a data-sparse approximation $\tilde{\mathbf{X}}^{(\ell)} \approx \mathbf{X}^{(\ell)}$.
In particular, we make use of the hierarchical tensor format introduced in \cite{HackKuehn09} and analyzed in \cite{Gras10}.

\subsection{Hierarchical Tensor Format}
In the following, we consider tensors $\mathbf{X}\in\R^\calJ$ of order $d\in\N$ over general product index sets $\calJ = \calJ_1\times\ldots\times\calJ_d$.
We first need the concept of the matricization of a tensor.
\begin{definition}
Let $\calD\isdef\{1,\ldots,d\}$. Given a subset $t\subset \calD$ with complement $[t]\isdef\calD\setminus t$, the \emph{matricization}
\[
    \calM_t : \R^\calJ \to \R^{\calJ_t} \otimes \R^{\calJ_{[t]}},\qquad \calJ_t\isdef\bigtimes_{i\in t} \calJ_i,\quad \calJ_{[t]}\isdef\bigtimes_{i\in [t]} \calJ_i,
\]
of a tensor $\mathbf{X} \in \R^\calJ$ is defined by its entries
\begin{equation*} 
    \calM_t(\mathbf{X})_{(j_i)_{i\in t}, (j_i)_{i\in [t]}} \isdef \mathbf{X}_{(j_1,\ldots,j_d)},\qquad (j_1,\ldots,j_d)\in\calJ.
\end{equation*}
\end{definition}
The subsets $t\subset \calD$ are organized in a binary \emph{dimension tree} $T_\calD$ with root $\calD$ such that
each node $t\in T_\calD$ is non-empty and each $t\in T_\calD$ with $\#t\geq 2$ is the disjoint union of its sons $t_1,t_2\in T_\calD$, cf. Figure \ref{fig:dimtree}.

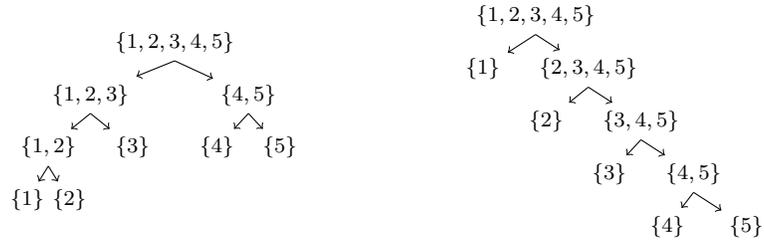
\begin{figure}[!ht]
\hspace{1cm}
\parbox{0.4\linewidth}{
\scalefont{0.75}
\centering
\begin{tikzpicture}[scale=0.7]
   \coordinate (T17) at (0,0);  
   \coordinate (T14) at (-1.6,-1);
   \coordinate (T57) at (1.4,-1);
   \coordinate (T12) at (-2.4,-2);
   \coordinate (T34) at (-0.8,-2);
   \coordinate (T56) at (0.8,-2);
   \coordinate (T7) at (2.0,-2);
   \coordinate (T1) at (-2.8,-3);
   \coordinate (T2) at (-2.0,-3);
   \coordinate (L0) at (-4,0);
   \coordinate (L1) at (-4,-1);
   \coordinate (L2) at (-4,-2);
   \coordinate (L3) at (-4,-3);
   \node (N17) at (T17) {$\{1,2,3,4,5\}$};
   \node (N14) at (T14) {$\{1,2,3\}$};
   \node (N57) at (T57) {$\{4,5\}$};
   \node (N12) at (T12) {$\{1,2\}$};
   \node (N34) at (T34) {$\{3\}$};
   \node (N56) at (T56) {$\{4\}$};
   \node (N7) at (T7) {$\{5\}$};
   \node (N1) at (T1) {$\{1\}$};
   \node (N2) at (T2) {$\{2\}$};
   \draw[->] (N17.south) -- (N14);
   \draw[->] (N17.south) -- (N57);  
   \draw[->] (N14.south) -- (N12);
   \draw[->] (N14.south) -- (N34);    
   \draw[->] (N57.south) -- (N56);  
   \draw[->] (N57.south) -- (N7);  
   \draw[->] (N12.south) -- (N1);
   \draw[->] (N12.south) -- (N2);
\end{tikzpicture}
}
\parbox{0.40\linewidth}{
\scalefont{0.75}
\centering
\begin{tikzpicture}[scale=0.7]
   \coordinate (T17) at (0,0);  
   \coordinate (T27) at (1.0,-1);
   \coordinate (T37) at (2.0,-2);
   \coordinate (T47) at (3.0,-3);
   \coordinate (T57) at (4.0,-4);
   \coordinate (T1) at (-1.0,-1);
   \coordinate (T2) at (0.2,-2);
   \coordinate (T3) at (1.4,-3);
   \coordinate (T4) at (2.5,-4);
   \coordinate (L0) at (6.5,0);
   \coordinate (L1) at (6.5,-1);
   \coordinate (L2) at (6.5,-2);
   \coordinate (L3) at (6.5,-3);
   \coordinate (L4) at (6.5,-4);
   \coordinate (L5) at (6.5,-5);
   \coordinate (L6) at (6.5,-6);
   \node (N17) at (T17) {$\{1,2,3,4,5\}$};
   \node (N27) at (T27) {$\{2,3,4,5\}$};
   \node (N37) at (T37) {$\{3,4,5\}$};
   \node (N47) at (T47) {$\{4,5\}$};
   \node (N57) at (T57) {$\{5\}$};
   \node (N1) at (T1) {$\{1\}$};
   \node (N2) at (T2) {$\{2\}$};
   \node (N3) at (T3) {$\{3\}$};
   \node (N4) at (T4) {$\{4\}$};
   \draw[->] (N17.south) -- (N1);
   \draw[->] (N17.south) -- (N27);  
   \draw[->] (N27.south) -- (N2);
   \draw[->] (N27.south) -- (N37);    
   \draw[->] (N37.south) -- (N3);  
   \draw[->] (N37.south) -- (N47);  
   \draw[->] (N47.south) -- (N4);
   \draw[->] (N47.south) -- (N57);
\end{tikzpicture}
}
\vspace{-0.2cm}
\caption{Dimension trees $T_\calD$ for $d=5$. Left: Balanced tree. Right: Linear tree.
\label{fig:dimtree}}
\end{figure}
%
%
\begin{definition}
\label{def:htensor}
Let $T_\calD$ be a dimension tree. 
The \emph{hierarchical rank} $\bs{r}\isdef(r_t)_{t\in T_\calD}$ of a tensor $\mathbf{X}\in\R^\calJ$ is defined by
\[
    r_t \isdef \operatorname{rank}\big(\calM_t(\mathbf{X})\big),\quad t\in T_\calD.
\]
For a given hierarchical rank $\bs{r}\isdef(r_t)_{t\in T_\calD}$, 
the \emph{hierarchical format} $\calH_{\bs{r}}$ is defined by
\[
     \calH_{\bs{r}} \isdef \big\{\mathbf{X}\in \R^\calJ : \operatorname{rank}\big(\calM_t(\mathbf{X})\big) \leq r_t,\, t\in T_\calD\big\}.
\]
\end{definition}
Given a tensor $\mathbf{X}\in\calH_{\bs{r}}$, Definition \ref{def:htensor} implies that for all $t\in T_\calD$ 
we can choose (orthogonal) matrices $\mathbf{U}_t \in \R^{\calJ_t\times r_t}$ 
such that $\operatorname{range}(\mathbf{U}_t) = \operatorname{range}\big(\calM_t(\mathbf{X})\big)$.
Moreover, for every non-leaf node $t \in T_\calD$ with sons $t_1,t_2\in T_\calD$, there exists a \emph{transfer tensor}
$\mathbf{B}_t\in\R^{r_t\times r_{t_1} \times r_{t_2}}$ such that
\begin{equation}\label{eq:nestedness}
  (\mathbf{U}_t)_{\cdot,s} = \sum_{s_1=1}^{r_{t_1}} \sum_{s_2 = 1}^{r_{t_2}} 
  (\mathbf{B}_t)_{(s,s_1,s_2)} (\mathbf{U}_{t_1})_{\cdot,s_1} \otimes (\mathbf{U}_{t_2})_{\cdot,s_2},\qquad s = 1,\ldots,r_t,
\end{equation}
where $(\mathbf{U}_t)_{\cdot,s}$ denotes the $s$th column of $\mathbf{U}_t$. 
At the root node $t=\calD$, we identify the tensor $\mathbf{X}$ with the column matrix $\mathbf{U}_\calD\in\R^{\calJ\times 1}$.

The recursive relation \eqref{eq:nestedness} is key to represent the tensor $\mathbf{X}$ compactly.
For all leaf nodes $t\in T_\calD$, we explicitly store the matrices $\mathbf{U}_t$, whereas for all inner nodes $t\in T_\calD$ only the 
transfers tensors $\mathbf{B}_t$ are required. The complexity for the hierarchical representation sums up to $\calO(dr^3 + drn)$, where $r\isdef r_\text{max} = \max_{t\in T_\calD} r_t$, $n\isdef\max_{i\in \calD} \#\calJ_i$. 
The \emph{effective rank} $r_\mathrm{eff}$ is the real positive number such that $(d-1)r_\mathrm{eff}^3 + dr_\mathrm{eff}n$ is the actual storage cost for a tensor in $\calH_{\bs{r}}$. 

In the multilevel scheme introduced above, the tensor $\mathbf{X}^{(\ell)}$ from \eqref{eq:defX} is defined via the 
numerical solution of the original PDE on levels $\ell$ and $\ell-1$ at all collocation points.
This means that an explicit computation of $\mathbf{X}^{(\ell)}$ in terms of all its entries would only be possible
for small length \(N\) of the Karhunen-Lo\`eve expansion and moderate polynomial degrees $\bs{p}^{(\ell)}$.
To overcome this limitation, we suggest to approximate $\mathbf{X}^{(\ell)}$ directly in the hierarchical tensor format $\calH_{\bs{r}}$
by the so-called \emph{cross approximation} technique introduced in \cite{BalGrasKlug2010}.

\subsection{Cross Approximation Technique}
The main idea of tensor cross approximation is to exploit the inherent low-rank structure directly by the evaluation of
a (small) number of well-chosen tensor entries. Prior numerical experiments indicate that cross approximation 
works particularly well for tensors of small size in each direction $i=1,\ldots,d$.
Considering the tensor $\mathbf{X} = \mathbf{X}^{(\ell)}$ from \eqref{eq:defX}, we observe that the size $n_\ell$ in direction $d=N+1$
becomes rather large for higher levels $\ell$ such that the cross approximation technique cannot be applied directly. We therefore use the following 
variant consisting of three steps:
\begin{description}
\item[Step 1.] Find an (orthogonal) matrix $\mathbf{V}\in\R^{n_\ell \times r_d}$ such that
\[
    \calM_{\{d\}}(\mathbf{X}) \approx \mathbf{V}\mathbf{V}^\top \calM_{\{d\}}(\mathbf{X}).
\]
\item[Step 2.]  Define a tensor $\mathbf{Y}\in\R^{\calJ^\prime}$ with $\calJ^\prime \isdef \calJ_{\{1,\ldots,d-1\}}\times \{1,\ldots,r_d\}$ via
\begin{equation}\label{eq:defY}
    \calM_{\{d\}}(\mathbf{Y}) = \mathbf{V}^\top \calM_{\{d\}}(\mathbf{X}).
\end{equation}
and use cross approximation to find $\tilde{\mathbf{Y}} \approx \mathbf{Y}$.
\item[Step 3.]  Build the final approximation $\tilde{\mathbf{X}}$ from
\begin{equation}\label{eq:defXtilde}
    \calM_{\{d\}}(\tilde{\mathbf{X}}) = \mathbf{V} \calM_{\{d\}}(\tilde{\mathbf{Y}}).
\end{equation}
\end{description}
The advantage of applying the cross approximation technique to the tensor $\mathbf{Y}$ instead of $\mathbf{X}$ lies in the reduced size
in direction $d=N+1$ for which we expect $r_d \ll n_\ell$. We now describe in more detail how the three approximation steps are carried out.

In \textbf{\textsf{Step 1}}, our aim is to construct an (approximate) basis $\mathbf{V}$ of the column space of $\calM_{\{d\}}(\mathbf{X})$.
To this end, we use the greedy strategy from Algorithm \ref{alg:columnspace}
over a subset $\calJ_\mathrm{train} \subset \calJ_{\{1,\ldots,d-1\}}$ of column indices.
\begin{algorithm}[!ht]
\caption{Find column basis $\mathbf{V}$ of $\calM_{\{d\}}(\mathbf{X})$}
\label{alg:columnspace}
\begin{algorithmic}[1]
\STATE{$\mathbf{V}\isdef[\,]$}
\REPEAT
\STATE{$\bs{j}^* \isdef \arg\max_{\bs{j}\in\calJ_\mathrm{train}} \big\| (\mathbf{I}-\mathbf{V}\mathbf{V}^\top)\calM_{\{d\}}(\mathbf{X})_{\cdot,\bs{j}}\big\|_2$} \label{ln:line3}
\STATE{$\mathbf{V} \isdef \text{orth}[\mathbf{V}, \calM_{\{d\}}(\mathbf{X})_{\cdot,\bs{j}^*}]$} 
\UNTIL{$\big\| (\mathbf{I}-\mathbf{V}\mathbf{V}^\top)\calM_{\{d\}}(\mathbf{X})_{\cdot,\bs{j}^*} \big\|_2 \leq \varepsilon$} 
\end{algorithmic}
\end{algorithm}
To construct the training set $\calJ_\mathrm{train}$, we use the following strategy known from tensor cross approximation~\cite[Sec. 3.5]{GrasKresTob13}. 
Starting with a random index $\bs{j}\in \calJ_{\{1,\ldots,d-1\}}$, we consider the set
\begin{equation}\label{eq:cross_idx}
    \calJ_\text{cross}(\bs{j}) \isdef \{ (j_1,\ldots,j_{i-1},k,j_{i+1},\ldots,j_{d-1}): k\in\calJ_i,\, i = 1,\ldots,d-1\} 
\end{equation}
which forms a 'cross' with center $\bs{j}$. 
Repeating this strategy a few number $s$ of times (say $s=3$) for random indices $\bs{j}^1,\ldots,\bs{j}^s\in\calJ_{\{1,\ldots,d-1\}}$, we arrive at 
\[
    \calJ_\mathrm{train} \isdef \calJ_\mathrm{cross}(\bs{j}^1)\cup\ldots\cup \calJ_\mathrm{cross}(\bs{j}^s),
\]    
which determines the training set for the first loop of Algorithm~\ref{alg:columnspace}:
In every subsequent loop of Algorithm~\ref{alg:columnspace}, this set is enriched with $s$ additional (random) crosses.
In line~\ref{ln:line3}, we reuse the information computed in the previous loops of the algorithm as much as possible.

Once the matrix $\mathbf{V}$ is constructed, our next aim in \textbf{\textsf{Step 2}} is to approximate the tensor $\mathbf{Y}\in\R^{\calJ^\prime}$ from \eqref{eq:defY}
in the hierarchical tensor format $\calH_{\bs{r}}$. Recalling the main idea of the approach in \cite{BalGrasKlug2010},
we seek to recursively approximate the matricizations of $\mathbf{M} = \calM_t(\mathbf{Y})$
at any node $t\in T_\calD$ by a so-called \emph{cross approximation} of the form
\begin{equation}\label{eq:cross}
    \mathbf{M} \approx \tilde{\mathbf{M}} \isdef \mathbf{M}\rvert_{\calJ_t^\prime\times\calC_t} \cdot \mathbf{M}\rvert_{\calR_t\times\calC_t}^{-1} \cdot \mathbf{M}\rvert_{\calR_t\times \calJ_{[t]}^\prime}
\end{equation}
with $\op{rank}(\tilde{\mathbf{M}}) = r_t$ and pivot sets $\calR_t \subset \calJ_t^\prime$, $\calC_t \subset \calJ_{[t]}^\prime$ of size $r_t$.
For each node $t\in T_\calD$, the rank $r_t$ can be chosen adaptively in order to reach a given (heuristic) target accuracy 
$\varepsilon_\text{ten}\geq 0$ such that $\|\mathbf{M}-\tilde{\mathbf{M}}\|_2 \approx \varepsilon_\text{ten} \|\mathbf{M}\|_2$.

The matrices $\mathbf{M}\rvert_{\calJ_t^\prime\times\calC_t}, \mathbf{M}\rvert_{\calR_t\times \calJ_{[t]}^\prime}$ in~\eqref{eq:cross} are never formed
explicitly. The essential information for the construction of $\mathbf{Y}\in\calH_{\bs{r}}$ with $\bs{r}=(r_t)_{t\in T_\calD}$
are condensed in the pivot sets $\calR_t,\calC_t$ and the matrices $\mathbf{M}\rvert_{\calR_t\times\calC_t}\in\R^{r_t\times r_t}$ from \eqref{eq:cross}. 
This construction is explicit in the sense that the necessary transfer tensors $\mathbf{B}_t$ for all inner nodes $t\in T_\calD$ and the matrices $\mathbf{U}_t$ in the
leaf nodes $t\in T_\calD$ are directly determined by the values of $\mathbf{Y}$ at certain entries defined by the pivots sets.
The details of this procedure can be found in \cite{BalGrasKlug2010,Ha2012}.

%
%
%
After the cross approximation has been performed, \textbf{\textsf{Step 3}} involves no further approximation but only a simple matrix-matrix product.
Assume that the tensor $\mathbf{Y}$ has been approximated by $\tilde{\mathbf{Y}}$ represented in $\calH_{\bs{r}}$ by means of transfer tensors
$\mathbf{B}_t$ for inner nodes $t\in T_\calD$ and matrices $\mathbf{U}_t$ for leaf nodes $t\in T_\calD$.
In the node $t=\{d\}$, we now compute the matrix $\mathbf{U}_t^\prime \isdef \mathbf{V}\mathbf{U}_t$, whereas for all other leaf nodes $t\in T_\calD$
we keep $\mathbf{U}_t^\prime \isdef \mathbf{U}_t$. It turns out that the tensor $\tilde{\mathbf{X}}$ from \eqref{eq:defXtilde} is then represented by
the transfer tensors $\mathbf{B}_t$ and the matrices $\mathbf{U}_t^\prime$.

\subsection{Error Analysis}
We now study the effect of a perturbed multilevel approximation introduced through tensor approximations $\tilde{\mathbf{X}}^{(\ell)} \approx \mathbf{X}^{(\ell)}$.
In particular our aim is to derive an indication from Lemma \ref{lem:perturbed_ML} for the required accuracy in the tensor approximation in order to maintain the convergence result for the multilevel scheme.

Thanks to the orthogonality of the basis $\{\psi_{\ell,i}\}$ in \eqref{eq:delta_psi}, we immediately derive from \eqref{eq:defX} that
\[
  \norm{\delta_{N,\ell}(\bs{y}_{\bs{k}}^{(\ell)})}_{H_0^1(D)} = \norm{\mathbf{X}^{(\ell)}_{(\bs{k},\cdot)}}_2,\qquad \bs{k}\in\calK_{\bs{p}^{(\ell)}}.
\]
In order to apply Lemma \ref{lem:perturbed_ML}, we need to ensure that
\[
  \norm{\tilde{\mathbf{X}}^{(\ell)}_{(\bs{k},\cdot)}-\mathbf{X}^{(\ell)}_{(\bs{k},\cdot)}}_2 \lesssim 2^{-L},\qquad \bs{k}\in\calK_{\bs{p}^{(\ell)}}.
\]
Noting that $\norm{\delta_{N,\ell}(\bs{y}_{\bs{k}}^{(\ell)})}_{H_0^1(D)}\lesssim 2^{-\ell}$, this can be guaranteed if we require
\[
  \norm{\tilde{\mathbf{X}}^{(\ell)}_{(\bs{k},\cdot)}-\mathbf{X}^{(\ell)}_{(\bs{k},\cdot)}}_2 \lesssim
  2^{\ell-L} \norm{\mathbf{X}^{(\ell)}_{(\bs{k},\cdot)}}_2,\qquad \bs{k}\in\calK_{\bs{p}^{(\ell)}}.
\]
This motivates to perform the tensor approximation with a relative accuracy of $\varepsilon_\ell \sim 2^{\ell-L}$ such that
\begin{equation}\label{eq:tensor_acc}
    \norm{\tilde{\mathbf{X}}^{(\ell)} - \mathbf{X}^{(\ell)}}_2 \lesssim \varepsilon_\ell\norm{\mathbf{X}^{(\ell)}}_2.
\end{equation}
As a consequence, the tensor approximation for higher levels $\ell$ needs to be done less accurate.

\subsection{Final Algorithm}
Compiling all the results obtained so far, our final strategy is summarized in Algorithm \ref{alg:multilevel}.

\begin{algorithm}[!ht]
\caption{Multilevel tensor approximation \emph{ML-Tensor}}
\label{alg:multilevel}
\begin{algorithmic}[1]
\FOR{$\ell=0,\ldots,L$}
\STATE{Choose $h_\ell \sim 2^{-\ell}$, $p^{(\ell)} \sim L-\ell$, $\varepsilon_\ell \sim 2^{\ell-L}$}
\STATE{Approximate $\tilde{\mathbf{X}}^\ell \approx\mathbf{X}^\ell$ from \eqref{eq:defX} using Steps 1,2,3 with relative accuracy $\varepsilon_\ell$}
\ENDFOR
\RETURN{Multilevel approximation $\tilde{u}_L^{\mathrm{ML}}$ from \eqref{eq:perturbed_ML}}
\end{algorithmic}
\end{algorithm}

\section{Numerical Experiments} \label{sec:experiments}

In the numerical experiments, we consider the parametric diffusion problem on the unit square given by
\begin{align*}
  -\nabla\cdot\big(a(\bs{y})\nabla u(\bs{y})\big) &= 1,\quad \, \mbox{in}\,\, D = (0,1)^2,\\
  u(\bs{y}) &= 0,\quad\, \mbox{on}\,\, \partial D.
\end{align*}
On each level $\ell$ of the proposed multilevel scheme, the domain $D$ is discretized by a uniform triangulation 
with mesh size
\[
    h_\ell = 2^{-\ell}h_0,\qquad h_0 = 1/4,
\]
using \(\mathcal{Q}_1\), i.e., bilinear finite elements with $n_\ell$ degrees of freedom.

To construct the interpolation operator $\calI_{\bs{p}^{(\ell)}}$ from \eqref{eq:ipol_def}, we use an isotropic polynomial
degree on each level defined by
\[
    \bs{p}^{(\ell)} = (p^{(\ell)},\ldots,p^{(\ell)})\in\N_0^N,\qquad p^{(\ell)}\isdef \lfloor(L-\ell+1)/2\rfloor.
\]
This means that possible anisotropies induced by the decay of the Karhunen-Lo\`eve expansion are not considered here.
The interpolation points $\bs{y}_{\bs{k}} \in \Gamma = [-1,1]^N$ are given by the tensorized
roots of the Chebyshev polynomials of the first kind of degree $p^{(\ell)}+1$.
The accuracy for the tensor approximation from \eqref{eq:tensor_acc} on each level is chosen as
\[
    \varepsilon_\ell = 2^{\ell-L}\varepsilon_0,\qquad \varepsilon_0 = 1/4.
\]    
For each level $\ell$, we report the effective rank $r_\mathrm{eff}$ and the maximal rank $r_\mathrm{max}$ of the 
approximate tensor $\mathbf{X}^{(\ell)}$ represented in the hierarchical tensor format $\calH_{\bs{r}}$.
In addition, we state the number of tensor evaluations for \textbf{\textsf{Step 1}} and \textbf{\textsf{Step 2}} during the cross approximation procedure of the tensor $\mathbf{X}^{(\ell)}$.
Note that each evaluation on level $\ell$ may require the solution of the PDE on level $\ell$ and level $\ell-1$.

To measure the interpolation error, we randomly choose $M=100$ parameters $\bs{y}^{i}\in\Gamma$ and compute
\[
    \varepsilon^{\operatorname{M\!L}}_L [u]\isdef \left(\sum_{i=1}^M \norm{\tilde{u}_{N,L}^{\operatorname{M\!L}}(\bs{y}^{i}) - u_{N,L}(\bs{y}^{i})}_{H_0^1(D)}^2 \bigg/ \sum_{i=1}^M \norm{u_{N,L}(\bs{y}^{i})}_{H_0^1(D)}^2\right)^{1/2}.
\]
To study the impact of the different levels, we also compute for the perturbed differences $\tilde{\delta}_{N,\ell}$ the error
\[
    \varepsilon^{(\ell)}_L [u]\isdef \left(\sum_{i=1}^M \norm{\calI_{\bs{p}^{(\ell)}}[\tilde{\delta}_{N,\ell}](\bs{y}^{i})}_{H_0^1(D)}^2 \bigg/ \sum_{i=1}^M \norm{u_{N,L}(\bs{y}^{i})}_{H_0^1(D)}^2\right)^{1/2}.
\]
For a uniform distribution of $y_n \sim \calU([-1,1])$, $n=1,\ldots,N$, we evaluate the expected value of the multilevel solution and compute
\[
    \varepsilon^\mathbb{E}_L [u]\isdef \norm{\mathbb{E}\left[\tilde{u}_{N,L}^{\operatorname{M\!L}}\right]-\mathbb{E}[u_\mathrm{ref}]}_{H_0^1(D)} \Big/ \norm{\mathbb{E}[u_\mathrm{ref}]}_{H_0^1(D)},
\]
where $u_\mathrm{ref}$ is the reference solution obtained from the multilevel scheme on the highest level $L=7$.

From the multilevel solution, we can immediately compute approximations to output functionals, as, e.g., for
\[
  \psi(u)]\isdef \int_D u\de x.
\] 
Analogous to the errors for the solution, we then obtain relative errors for the output functional and for the expected value as
\[
    \varepsilon^\mathbb{E}_L [\psi] \isdef \left|\mathbb{E}\left[\psi\left(\tilde{u}_{N,L}^{\operatorname{M\!L}}\right)\right]-\mathbb{E}[\psi(u_\mathrm{ref})]\right| \Big/ \big|\mathbb{E}[\psi(u_\mathrm{ref})]\big|.
\]

All numerical experiments have been carried out on a quad-core Intel(R) Xeon(R) CPU E31225 with 3.10GHz. 
The timings spent on each level $\ell$ are CPU times for a single core. For the finite element approximation, we have used the software library deal.II, see \cite{BaHaKa2007}.
All sparse linear systems have been solved by a multifrontal solver from UMFPACK.

\subsection{Karhunen-Lo\`eve expansion with exponential decay}
In the first experiment, the Karhunen-Lo\`eve expansion of the diffusion coefficient is given by
\begin{equation}\label{eq:KL1}
    a(\bs{y},x) = 2 + \sum_{n=1}^N \sqrt{\lambda_n} b_n(x) y_n
\end{equation}
with
\[
   b_n(x) = \sin(2\pi n x_1) \sin(2\pi n x_2).
\]
We consider an exponential decay of the eigenvalues defined by $\lambda_n \isdef \operatorname{exp}(-n)$.
The results of this experiment for $N=10,20$ can be found in Table \ref{tab:exp_error} and Figure \ref{fig:exp_error}.
\begin{table}[!ht]
\centering
\scalefont{0.7}
\begin{tabular}{c|c|c|r|r|r|r|r|r|c}
$N$ & $\ell$ & $p^{(\ell)}$  & $n_\ell$ &  $r_\mathrm{eff}$ &   $r_\mathrm{max}$  & step 1 & step 2 & time[s] & $\varepsilon^{(\ell)}_L$ \\ \hline
 10 &     0 &   4  &      25 &    2.12 &    4 &  247 &    3802 &     1.7 &  3.99e-04\\       
    &     1 &   3  &      81 &    5.48 &   15 &  187 &    7129 &    11.0 &  4.24e-04\\       
   &      2 &   3  &     289 &   11.62 &   52 &  559 &   13839 &    72.2 &  6.07e-04\\       
   &      3 &   2  &    1089 &   14.71 &   79 &  262 &    8011 &   160.9 &  1.05e-03\\       
   &      4 &   2  &    4225 &   13.49 &   71 &  340 &    7562 &   554.0 &  6.93e-04\\       
   &      5 &   1  &   16641 &    7.61 &   32 &  225 &     879 &   305.5 &  1.13e-03\\       
   &      6 &   1  &   66049 &    5.59 &   20 &  179 &     775 &  1144.8 &  6.58e-04\\       
   &      7 &   0  &  263169 &    1.00 &    1 &    1 &       1 &    11.9 &  1.92e-03\\ \hline        
 20  &    0 &   4  &      25 &    1.89 &    4 &  487 &   12829 &     7.1 &  2.71e-04\\        
   &      1 &   3  &      81 &    4.47 &   16 &  367 &   22680 &    48.7 &  4.20e-04\\        
  &       2 &   3  &     289 &    9.96 &   56 & 1099 &   48864 &   360.6 &  4.39e-04\\        
  &       3 &   2  &    1089 &   12.72 &   87 &  422 &   40206 &  1135.6 &  1.08e-03\\        
  &       4 &   2  &    4225 &   11.74 &   80 &  516 &   35244 &  3662.6 &  8.16e-04\\        
  &       5 &   1  &   16641 &    6.90 &   39 &  442 &    8983 &  3836.8 &  1.06e-03\\        
  &       6 &   1  &   66049 &    4.95 &   24 &  316 &    5870 & 10698.2 &  6.93e-04\\        
  &       7 &   0  &  263169 &    1.00 &    1 &    1 &       1 &    16.2 &  1.84e-03\\      
\end{tabular}
\caption{Karhunen-Lo\`eve expansion with exponential decay: Multilevel approximation for $L=7$ with the number of tensor evaluations for Step 1 and Step 2 and the time spent on each level.}
\label{tab:exp_error}
\end{table}
From the last column of Table~\ref{tab:exp_error}, it can be seen that our adaptive choice of polynomial degree and hierarchical ranks successfully equilibrate the error on the different finite element levels. The hierarchical ranks increase initially and then decrease again as the level increases. This decrease is the most important feature of our approach; it significantly reduces the cost, in terms of queries to the solution, on the finer levels and the overall solution process. Figure~\ref{fig:exp_error} shows that the error decreases proportionally with $h$ as the maximum number of levels increases, as expected from our error estimates.

\begin{figure}[!ht]
\parbox{0.49\linewidth}{\centering\includegraphics[scale=0.4]{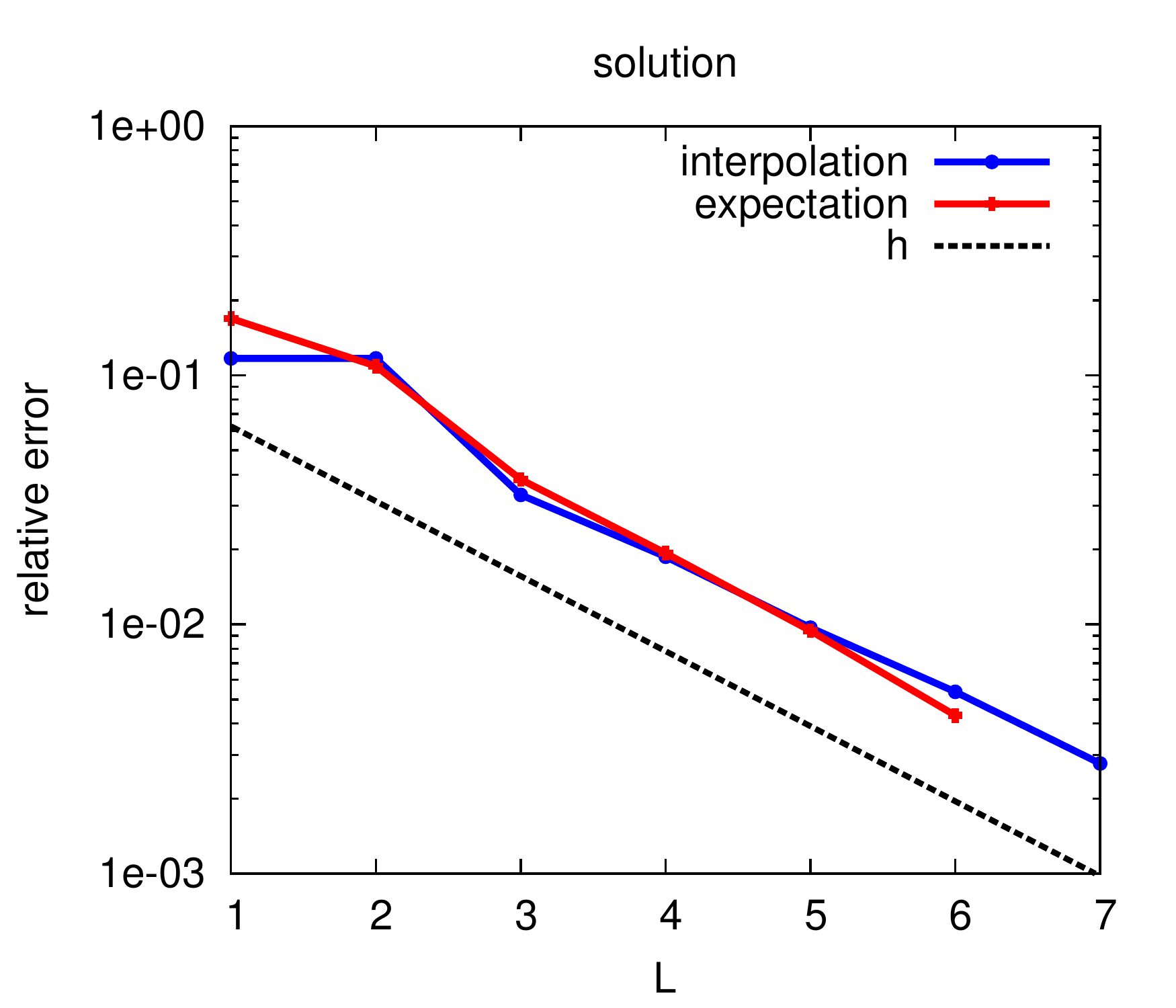}}
\parbox{0.49\linewidth}{\centering\includegraphics[scale=0.4]{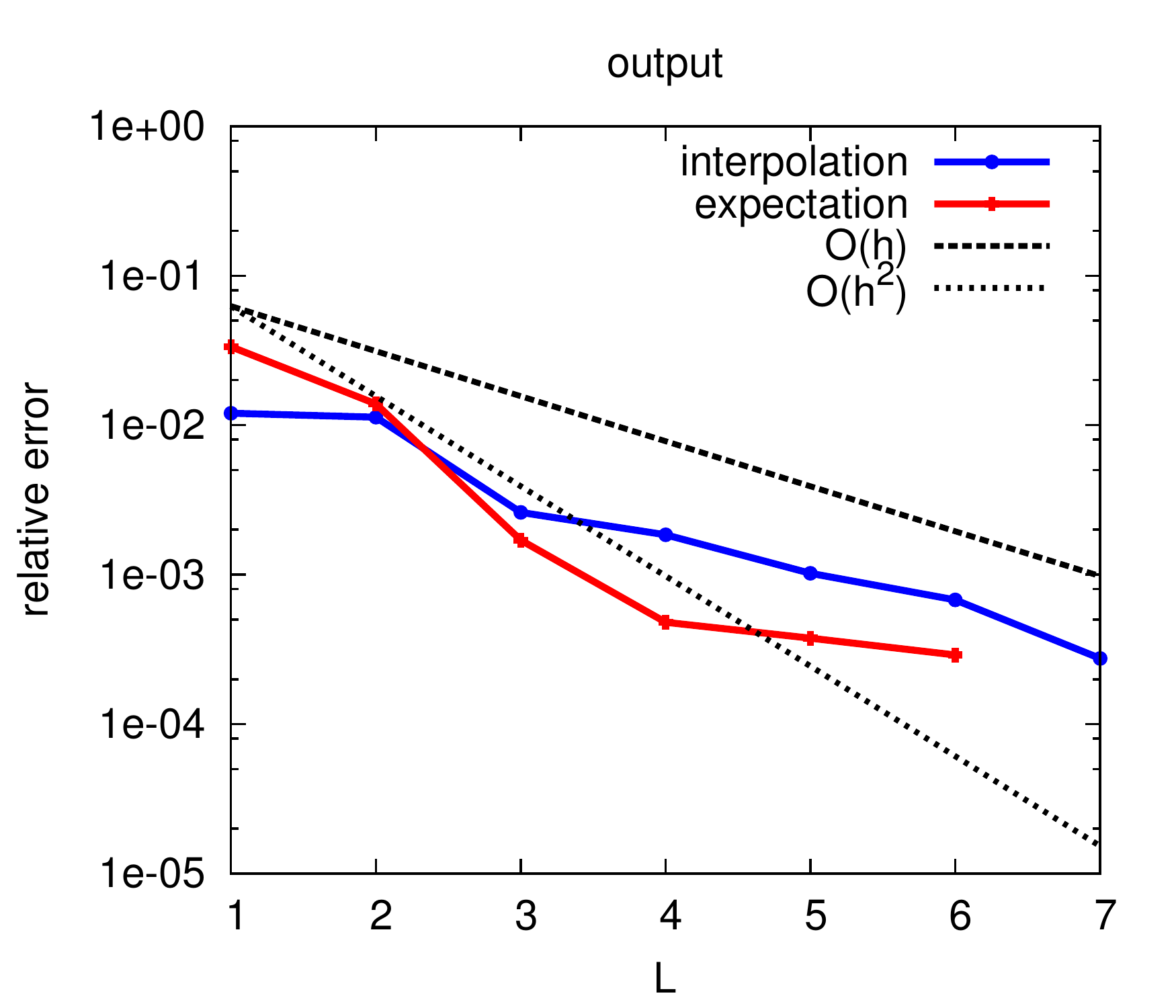}}
\caption{Karhunen-Lo\`eve expansion with exponential decay for $N=10$. Left: errors $\varepsilon^\mathrm{ML}_L [u]$ and $\varepsilon^\mathbb{E}_L [u]$ for the solution. 
Right: errors $\varepsilon^\mathrm{ML}_L [\psi]$ and $\varepsilon^\mathbb{E}_L [\psi]$ for the output functional.}
\label{fig:exp_error}
\end{figure}

\subsection{Karhunen-Lo\`eve expansion with fast algebraic decay}
In this experiment, the diffusion coefficient is again given by \eqref{eq:KL1}. 
We consider an algebraic decay of the eigenvalues defined by $\lambda_n \isdef 1/n^4$.
The results of this experiment for $N=10,20$ can be found in Table \ref{tab:alg_fast_error} and Figure \ref{fig:alg_fast_error}. 

\begin{table}[!ht]
\centering
\scalefont{0.7}
\begin{tabular}{c|c|c|r|r|r|r|r|r|c}
$N$ & $\ell$ & $p^{(\ell)}$  & $n_\ell$ &  $r_\mathrm{eff}$ &   $r_\mathrm{max}$  & step 1 & step 2 & time[s] & $\varepsilon^{(\ell)}_L$ \\ \hline
 10 &       0   &   4 &        25 &     1.84 &     4 &    247 &     2631 &      1.2 &  4.20e-04 \\         
    &       1   &   3 &        81 &     4.41 &    13 &    280 &     4492 &      7.2 &  4.42e-04 \\         
   &        2   &   3 &       289 &     6.87 &    28 &    559 &     5389 &     29.7 &  4.50e-04 \\         
   &        3   &   2 &      1089 &     7.49 &    31 &    378 &     3323 &     70.9 &  7.15e-04 \\         
   &        4   &   2 &      4225 &     6.34 &    23 &    378 &     2647 &    208.1 &  5.13e-04 \\         
   &        5   &   1 &     16641 &     4.60 &    14 &     67 &      814 &    239.2 &  6.98e-04 \\         
   &        6   &   1 &     66049 &     3.62 &    10 &    133 &      732 &   1004.6 &  5.53e-04 \\         
   &        7   &   0 &    263169 &     1.00 &     1 &      1 &        1 &     11.8 &  1.05e-03 \\ \hline  
 20  &      0   &   4 &        25 &     1.71 &     4 &    487 &     9458 &      5.3 &  3.51e-04 \\         
   &        1   &   3 &        81 &     3.66 &    14 &    367 &    16161 &     35.0 &  4.78e-04 \\         
  &         2   &   3 &       289 &     6.17 &    32 &   1099 &    26954 &    202.6 &  5.18e-04 \\         
  &         3   &   2 &      1089 &     7.93 &    46 &    862 &    21461 &    626.7 &  6.66e-04 \\         
  &         4   &   2 &      4225 &     7.96 &    43 &    862 &    32550 &   3454.5 &  4.69e-04 \\         
  &         5   &   1 &     16641 &     5.41 &    22 &    316 &    13526 &   5634.6 &  8.01e-04 \\         
  &         6   &   1 &     66049 &     3.65 &    13 &    379 &     6952 &  12580.2 &  5.52e-04 \\         
  &         7   &   0 &    263169 &     1.00 &     1 &      1 &        1 &     16.1 &  1.16e-03 \\        
\end{tabular}
\caption{Karhunen-Lo\`eve expansion with fast algebraic decay: Multilevel approximation for $L=7$ with the number of tensor evaluations for Step 1 and Step 2 and the time spent on each level}
\label{tab:alg_fast_error}
\end{table} 

\begin{figure}[!ht]
\parbox{0.49\linewidth}{\centering\includegraphics[scale=0.4]{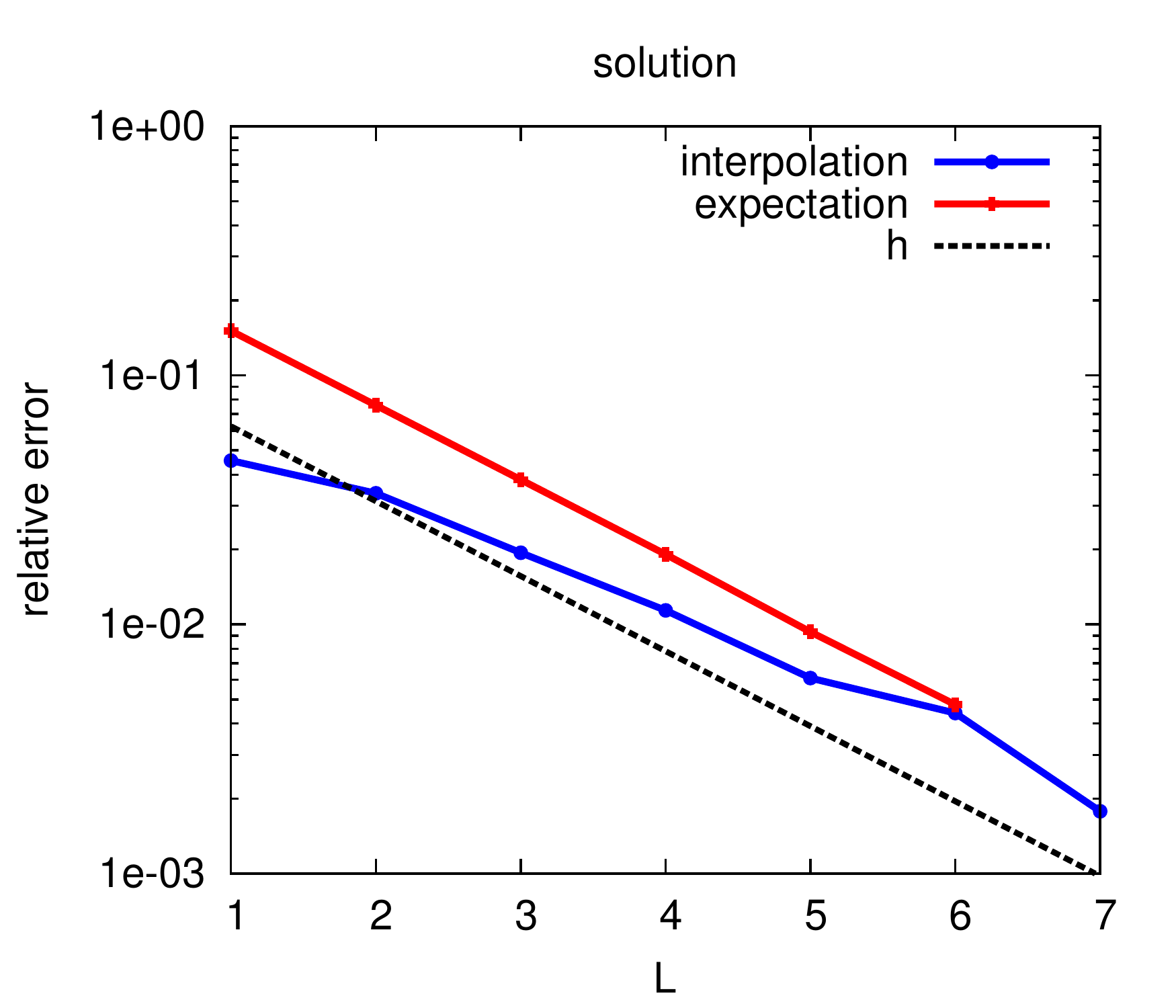}}
\parbox{0.49\linewidth}{\centering\includegraphics[scale=0.4]{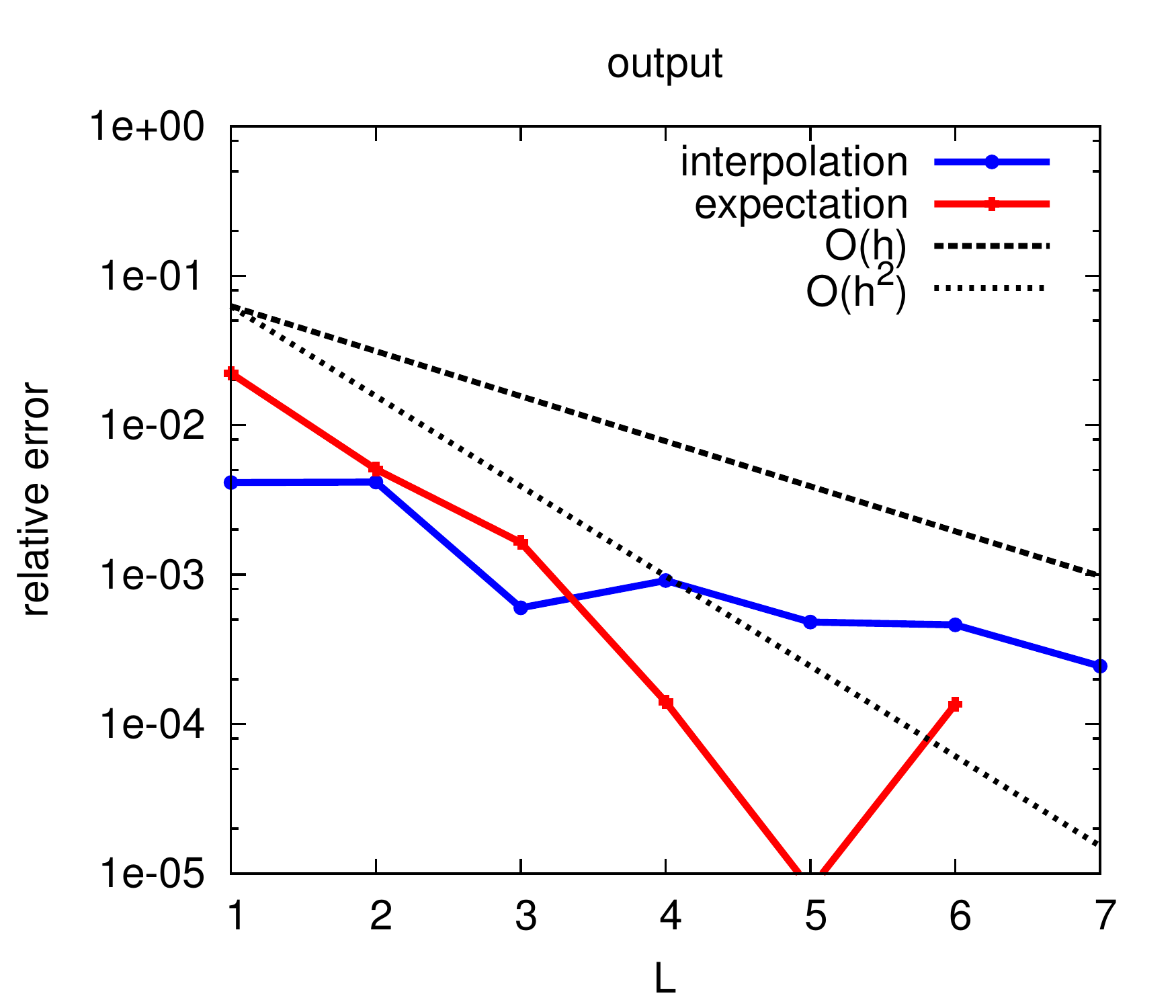}}
\caption{Karhunen-Lo\`eve expansion with fast algebraic decay for $N=10$.  Left: errors $\varepsilon^\mathrm{ML}_L [u]$ and $\varepsilon^\mathbb{E}_L [u]$ for the solution. 
Right: errors $\varepsilon^\mathrm{ML}_L [\psi]$ and $\varepsilon^\mathbb{E}_L [\psi]$ for the output functional.}
\label{fig:alg_fast_error}
\end{figure}

\subsection{Karhunen-Lo\`eve expansion with slow algebraic decay}

In this experiment, the diffusion coefficient is also given by \eqref{eq:KL1}. 
We consider an algebraic decay of the eigenvalues defined by $\lambda_n \isdef 1/n^2$.
The results of this experiment for $N=10$ can be found in Table \ref{tab:alg_slow_error} and Figure \ref{fig:alg_slow_error}. As expected, the maximal hierarchical rank becomes significantly higher compared to the faster algebraic decay. 

\begin{table}[!ht]
\centering
\scalefont{0.7}
\begin{tabular}{c|c|c|r|r|r|r|r|r|c}
$N$ & $\ell$ & $p^{(\ell)}$  & $n_\ell$ &  $r_\mathrm{eff}$ &   $r_\mathrm{max}$  & step 1 & step 2 & time[s] & $\varepsilon^{(\ell)}_L$ \\ \hline
10 &      0 &   4 &       25 &     2.36 &     4 &     247 &     4153 &      1.8  &   5.87e-04\\  
   &      1 &   3 &       81 &     5.82 &    16 &     187 &     9050 &     13.9  &   4.37e-04\\  
  &       2 &   3 &      289 &    13.64 &    58 &     373 &    19729 &    101.3  &   6.28e-04\\  
  &       3 &   2 &     1089 &    18.69 &    95 &     211 &    13083 &    259.8  &   2.35e-03\\  
  &       4 &   2 &     4225 &    17.63 &    93 &     209 &    11718 &    828.7  &   1.88e-03\\  
  &       5 &   1 &    16641 &    12.46 &    67 &     153 &     1012 &    330.0  &   1.92e-03\\  
  &       6 &   1 &    66049 &     9.14 &    41 &     323 &      961 &   1578.5  &   8.59e-04\\  
  &       7 &   0 &   263169 &     1.00 &     1 &       1 &        1 &     12.2  &   2.51e-03\\
\end{tabular}
\caption{Karhunen-Lo\`eve expansion with slow algebraic decay: Multilevel approximation for $L=7$ with the number of tensor evaluations for Step 1 and Step 2 and the time spent on each level}
\label{tab:alg_slow_error}
\end{table} 

\begin{figure}[!ht]
\parbox{0.49\linewidth}{\centering\includegraphics[scale=0.4]{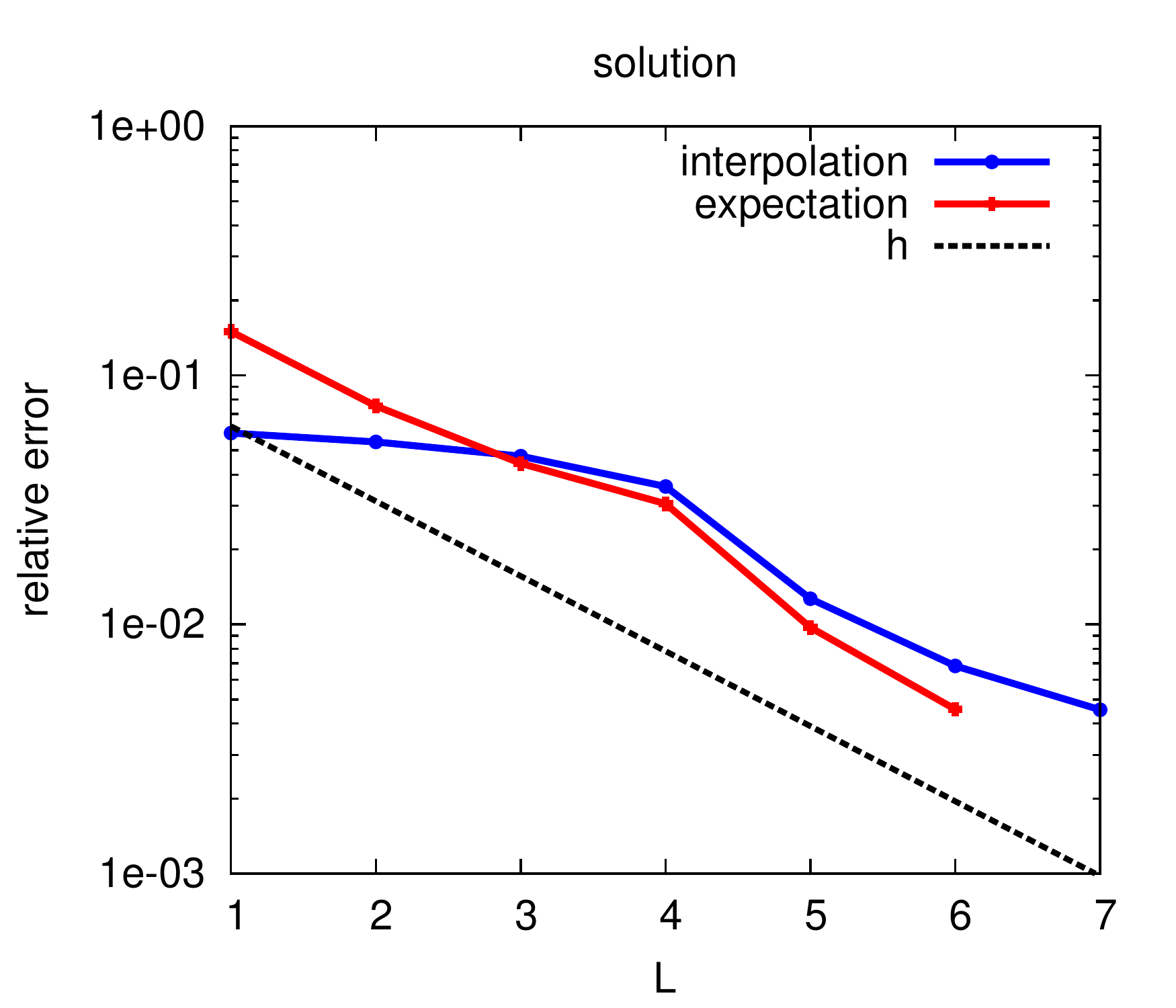}}
\parbox{0.49\linewidth}{\centering\includegraphics[scale=0.4]{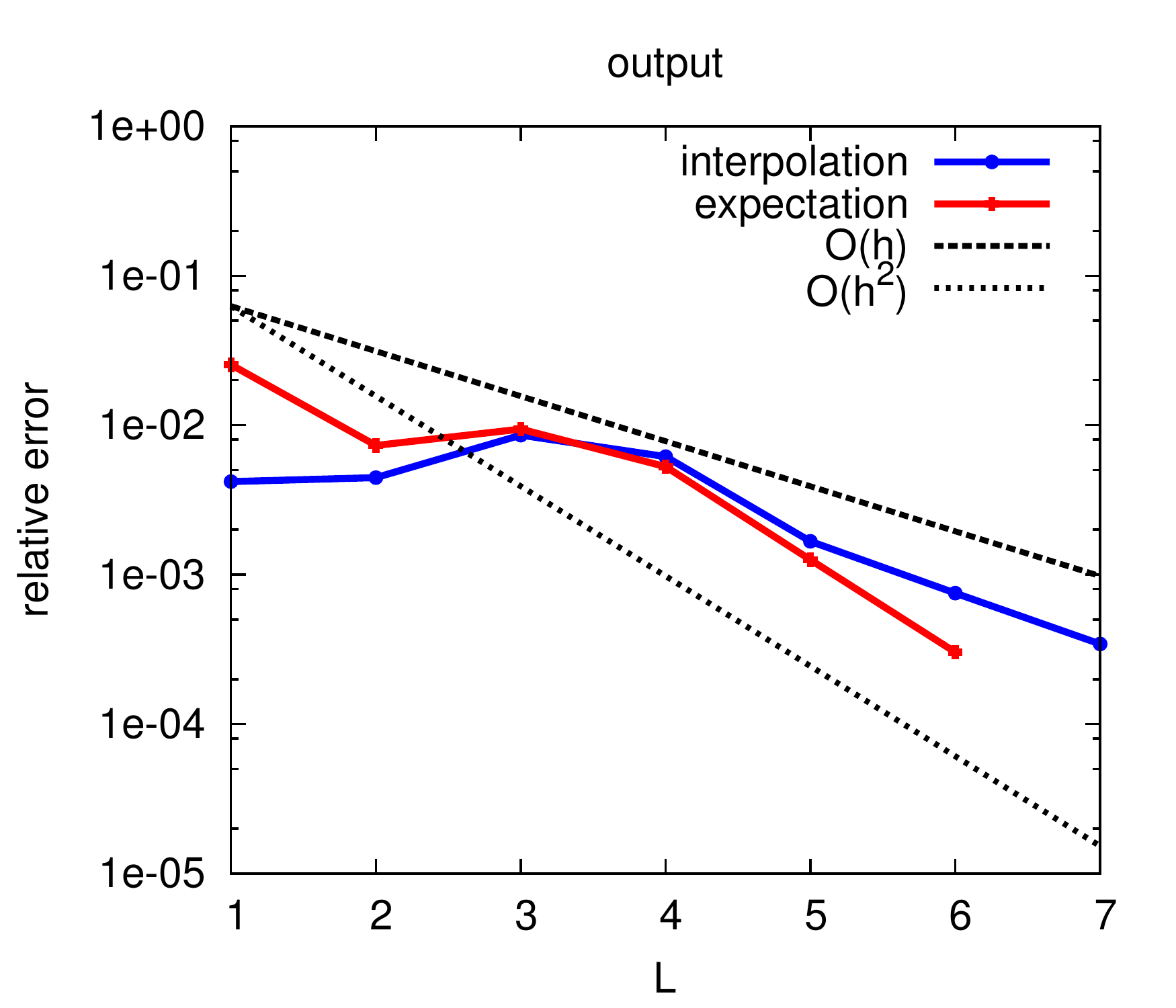}}
\caption{Karhunen-Lo\`eve expansion with slow algebraic decay for $N=10$. Left: errors $\varepsilon^\mathrm{ML}_L [u]$ and $\varepsilon^\mathbb{E}_L [u]$ for the solution. 
Right: errors $\varepsilon^\mathrm{ML}_L [\psi]$ and $\varepsilon^\mathbb{E}_L [\psi]$ for the output functional.}
\label{fig:alg_slow_error}
\end{figure}

\subsection{Log-uniform case}
Finally, to demonstrate that our approach does not depend on an affine linear decomposition of the diffusion coefficient with respect to the parameters,
we consider
\begin{equation*}
    a(\bs{y},x) = \exp\left( \sum_{n=1}^N \sqrt{\lambda_n} b_n(x) y_n \right),
\end{equation*}
with an algebraic decay defined by $\lambda_n \isdef 1/n^2$.
The results of this experiment for $N=10$ can be found in Table \ref{tab:loguniform_error} and Figure \ref{fig:loguniform_error}.

\begin{table}[!ht]
\centering
\scalefont{0.7}
\begin{tabular}{c|c|c|r|r|r|r|r|r|c}
$N$ & $\ell$ & $p^{(\ell)}$  & $n_\ell$ &  $r_\mathrm{eff}$ &   $r_\mathrm{max}$  & step 1 & step 2 & time[s] & $\varepsilon^{(\ell)}_L$ \\ \hline
10 &    0&   4 &       25  &    4.85 &    11 &    247 &    31106 &     13.0 & 6.40e-04  \\ 
   &    1&   3 &       81  &    9.65 &    28 &    187 &    36430 &     56.4 & 7.60e-04  \\ 
  &     2&   3 &      289  &   15.37 &    60 &    373 &    51920 &    270.3 & 1.03e-03  \\ 
  &     3&   2 &     1089  &   17.10 &    90 &    211 &    19827 &    402.4 & 1.80e-03  \\ 
  &     4&   2 &     4225  &   15.20 &    80 &    222 &    13206 &    975.7 & 1.50e-03  \\ 
  &     5&   1 &    16641  &    9.96 &    46 &    182 &     1001 &    341.0 & 2.24e-03  \\ 
  &     6&   1 &    66049  &    7.34 &    28 &    118 &      941 &   1299.9 & 1.42e-03  \\ 
  &     7&   0 &   263169  &    1.00 &     1 &      1 &        1 &     12.4 & 4.64e-03  \\
\end{tabular}
\caption{Log-uniform case: Multilevel approximation for $L=7$ with the number of tensor evaluations for Step 1 and Step 2 and the time spent on each level}
\label{tab:loguniform_error}
\end{table} 

\begin{figure}[!ht]
\parbox{0.49\linewidth}{\centering\includegraphics[scale=0.4]{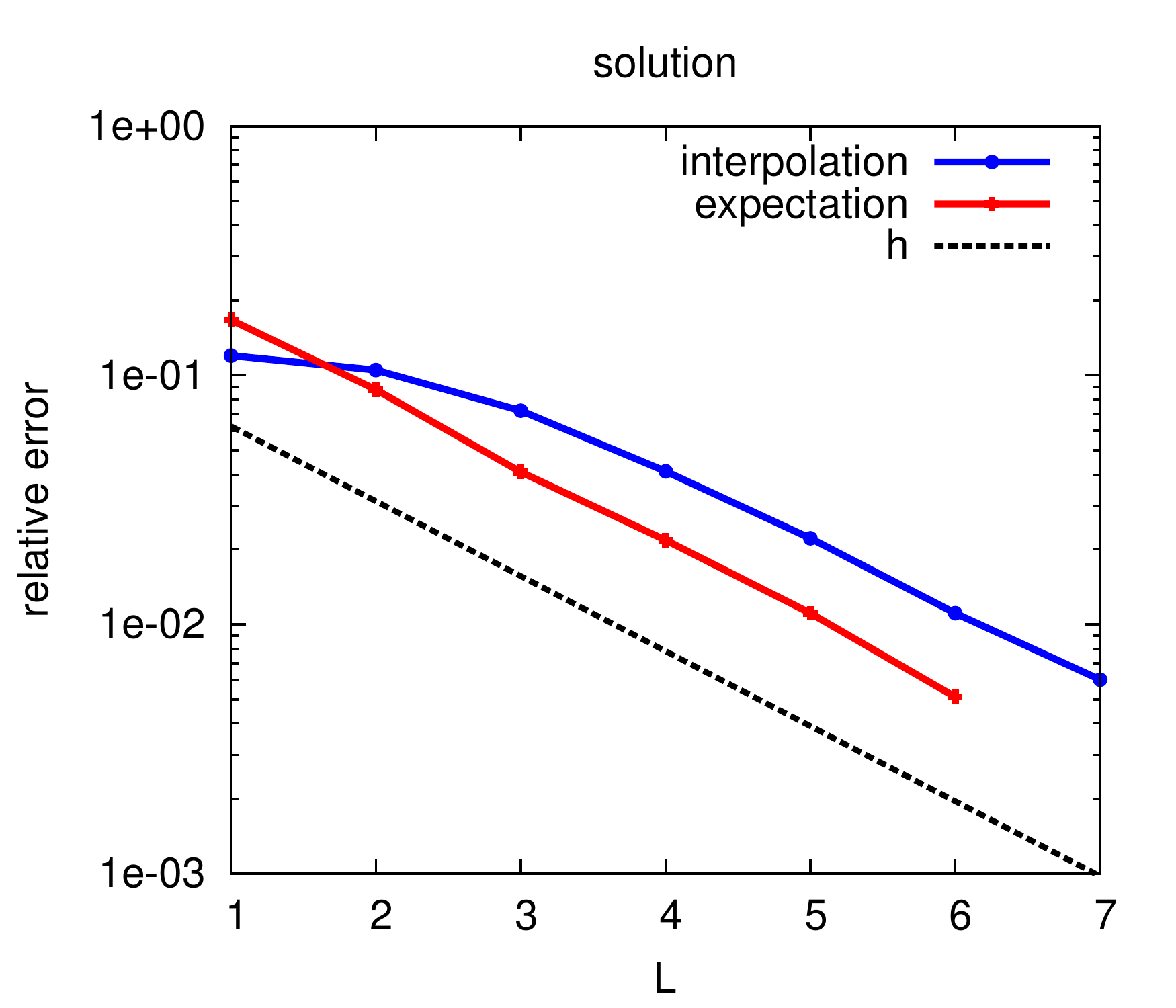}}
\parbox{0.49\linewidth}{\centering\includegraphics[scale=0.4]{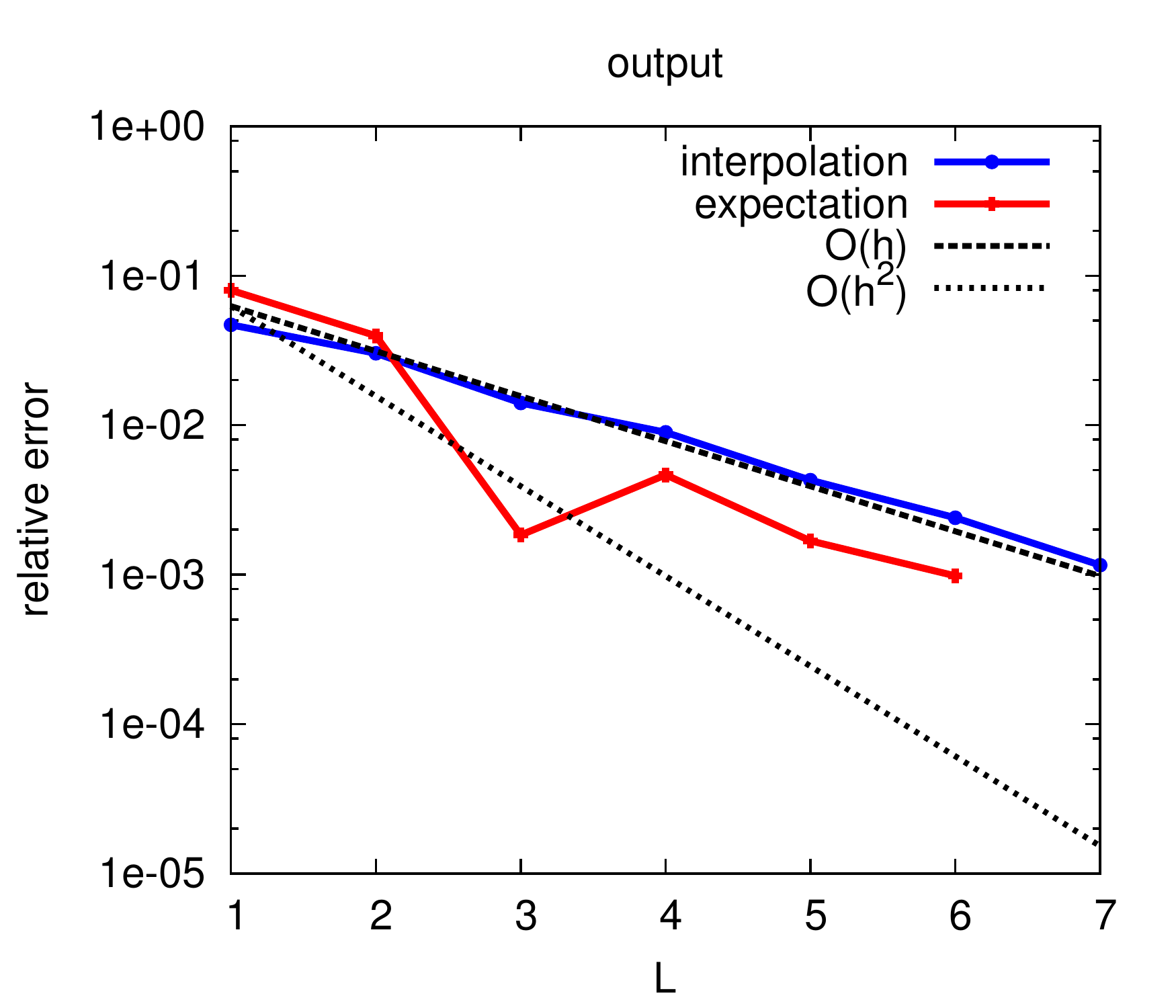}}
\caption{Log-uniform case for $N=10$.  Left: errors $\varepsilon^\mathrm{ML}_L [u]$ and $\varepsilon^\mathbb{E}_L [u]$ for the solution. 
Right: errors $\varepsilon^\mathrm{ML}_L [\psi]$ and $\varepsilon^\mathbb{E}_L [\psi]$ for the output functional.}
\label{fig:loguniform_error}
\end{figure}

\section{Conclusions}
In this article, we have considered the multilevel tensor approximation
for elliptic partial differential equations with a random diffusion coefficient.
By combining the multilevel idea for the approximation in the random parameter,
which has firstly been introduced in the context of multilevel Monte Carlo methods,
with a hierarchical tensor product approximation, we provide an efficient means
to directly represent the solution in a data sparse format. 
This representation can directly be employed for the evaluation of various functionals
of the solution without the necessity of performing additional costly computations.
In contrast to previous works, we do not 
rely on an a priori sparsified representation based on polynomials, but adaptively
compute a data sparse representation of the solution with the aid of the hierarchical
tensor format and the cross approximation. The numerical results confirm
the effectiveness of the presented method.

\bibliographystyle{plain}
\bibliography{literature}

\end{document}